\documentclass[11pt]{article}

\usepackage[T1]{fontenc}

\usepackage{blindtext}
\usepackage{enumitem} 
\setlist[enumerate]{parsep=0pt}
 
\usepackage{ragged2e} 

\usepackage{mathrsfs}
\usepackage{amssymb}
\usepackage{layout}
\usepackage{graphicx}
\usepackage{marvosym}
\usepackage[usenames,dvipsnames]{color}
\usepackage{verbatim}
\usepackage{enumitem}
\usepackage[pdftex]{hyperref}
\usepackage{fancyhdr}
\usepackage{indentfirst}
\usepackage{centernot}
\usepackage[all]{xy}

\usepackage{circuitikz}

\usepackage{amsthm,etoolbox}

\makeatletter
\@addtoreset{theorem}{section}
\@addtoreset{theorem}{subsection}
\@addtoreset{theorem}{subsubsection}

\usepackage{amsmath}

\usepackage{tikz-cd}
\usetikzlibrary{calc}

\usepackage{hyperref}
\usepackage{cleveref}[noabbrev]
\newtheorem{theorem}{Theorem}

\theoremstyle{definition}

\newtheorem{definition}[theorem]{Definition}

\theoremstyle{theorem}

\newtheorem{lemma}[theorem]{Lemma}
\newtheorem{prop}[theorem]{Proposition}
\newtheorem{cor}[theorem]{Corollary}

\newtheorem{remark}[theorem]{Remark}

\crefname{theorem}{Theorem}{Theorems}
\crefname{lemma}{Lemma}{Lemmas}
\crefname{prop}{Proposition}{Propositions}
\crefname{fact}{Fact}{Facts}
\crefname{remark}{Remark}{Remarks}
\crefname{cor}{Corollary}{Corollaries}

\newcommand{\theoremprefix}{}
\let\thetheoremsaved\thetheorem
\renewcommand{\thetheorem}{\theoremprefix\thetheoremsaved}

\patchcmd{\@startsection}{\par}{\renewcommand{\theoremprefix}{\csname the#1\endcsname.}}{}{}

\makeatother

\begin{document}

\clearpage
\pagenumbering{arabic}

\def\tp{\mbox{\rm tp}}
\def\qftp{\mbox{\rm qftp}}
\def\cb{\mbox{\rm cb}}
\def\wcb{\mbox{\rm wcb}}
\def\Diag{\mbox{\rm Diag}}
\def\trdeg{\mbox{\rm trdeg}}
\def\Gal{\mbox{\rm Gal}}
\def\Lin{\mbox{\rm Lin}}

\def\restriction#1#2{\mathchoice
              {\setbox1\hbox{${\displaystyle #1}_{\scriptstyle #2}$}
              \restrictionaux{#1}{#2}}
              {\setbox1\hbox{${\textstyle #1}_{\scriptstyle #2}$}
              \restrictionaux{#1}{#2}}
              {\setbox1\hbox{${\scriptstyle #1}_{\scriptscriptstyle #2}$}
              \restrictionaux{#1}{#2}}
              {\setbox1\hbox{${\scriptscriptstyle #1}_{\scriptscriptstyle #2}$}
              \restrictionaux{#1}{#2}}}
\def\restrictionaux#1#2{{#1\,\smash{\vrule height .8\ht1 depth .85\dp1}}_{\,#2}} 

\newcommand{\forkindep}[1][]{%
  \mathrel{
    \mathop{
      \vcenter{
        \hbox{\oalign{\noalign{\kern-.3ex}\hfil$\vert$\hfil\cr
              \noalign{\kern-.7ex}
              $\smile$\cr\noalign{\kern-.3ex}}}
      }
    }\displaylimits_{#1}
  }
}
\newpage
\begin{center}

\large \MakeUppercase{Looking for stabilizers in NSOP$_1$}

\vspace{5mm}

    \large Yvon \textsc{Bossut}
\end{center}
\vspace{10mm}

\vspace{10pt}
Abstract: In this work we study some examples of groups definable and type-definable in NSOP$_1$ theories. We exhibit some behaviors of these groups that differ from the ones of simple groups. We take interest in the notions of generics and stabilizers, and define the Kim-stabilizer.

\vspace{10pt}
We apply the notion of Kim-stabilizer and the stabilizer from \cite{hrushovski2011stablegrouptheoryapproximate} to the context of a group $G$ definable in an NSOP$_1$ field $F$ satisfying some assumptions to show that there is a finite to one embedding of a type definable subgroup of $G$ of bounded index into an algebraic group over $F$. We then show that definable groups in $\omega$-free PAC fields of characteristic $0$ satisfy these conditions.

\vspace{10pt}
\textbf{Acknowledgments:} I thank Frank O. Wagner for his advices and Nick Ramsey for his helpful emails.

\section{Introduction}

\vspace{10pt}
When we study a theory $T$ we want to understand its definable sets. In particular we want to analyze definable groups. We mention in this regard the results of Dobrowolski in \cite{dobrowolski2020sets} on groups definable in the theory of the generic bilinear form over an algebraically closed field and of Chatzidakis and Ramsey in \cite{chatzidakis2023measures} on groups definable in $e$-free PAC fields when $e$ is finite.

\vspace{10pt}
An effective result for this is the group configuration theorem, which Hrushovski proved in his thesis \cite{hrushovski1986contributions}. This result allows one to obtain a type-definable group from a certain configuration of independent elements in a stable theory. Typical applications of this result include:

\begin{theorem}\cite[Theorem C]{hrushovski1994groups} Let $G$ be a group definable in a pseudo-finite field $F$. Then $G$ is virtually definably isogenous to the set of $F$-rational points of an algebraic group over $F$.\end{theorem}

\vspace{10pt}
There are two major notions in the study of groups in model theory that we will now present in a rather informal way: the notion of formula and generic type in a group, and that of stabilizer of a type in a group. Let $G$ be a group definable in a model $\mathbb{M}$ of a theory $T$.

\vspace{10pt}
A definable subset $X$ of $G$ is said to be left-generic if a finite number of left translates of $X$ cover $G$, similarly for right-generic. A type is said to be (left/right) generic if it implies only (left/right) generic formulas. In a stable theory there exist generic complete types on any set of parameters, but this is not true in general. There is a similar notion in the context of simple theories. Given a complete type $p\in S_{G}(A)$ the stabilizer of $p$ is a subgroup $Stab(p) \leq G$ generated by elements 'stabilizing' the type, in the following sense: $g\in Stab(p)$ iff $p \cup g\cdot p$ is consistent, where $g\cdot p$ is the type $p$ translated by $g$. The link between these two notions is that a generic type must have a 'large' stabilizer - of bounded index for example.

\vspace{10pt}
The typical strategy for studying a definable group $G$ in a model $F$ of a field theory is as follows:

\begin{enumerate}
    \item Take $a,b,x$ generic elements (meaning that their type is generic) independent in the group $G$, write $c:=a\cdot b$.
    \item The group configuration theorem provides us with elements $a',b',c'$ in a group $H$ definable in the algebraic closure of $F^{alg}$ of $F$ (i.e., an algebraic group) interalgebraic with $a,b,c$ respectively on a set of parameters $E$.
    \item We then consider the stabilizer of $\tp(a,a'/E)$ in $G \times H$, and we show that it defines an isogeny between $G$ and $H$.
\end{enumerate}

\vspace{10pt}
This is the method that we want to apply to study groups definable in $\omega$-free PAC fields of characteristic $0$. To achieve this we seek to generalize the notions of genericity and stabilizer to NSOP$_1$. For the notion of generic, by taking the 'naive' generalization of the generic in the simple setting we have a negative answer (which was already noted by Dobrowolski: \cite[Proposition 8.14]{dobrowolski2020sets}). For the notion of stabilizer we give a definition generalizing that of the simple case, which defines in general an invariant group. We show that in certain examples the subgroup obtained, which we will call the Kim-stabilizer, is type-definable. We generalize some of the results on groups definable in a simple theory to this notion.

\vspace{10pt}
We also introduce another notion of stabilizer, due to Hrushovski in \cite{hrushovski2011stablegrouptheoryapproximate}. An ideal of formula is a set of formula closed under disjunction and logical implication (if $\theta \in \mu$ and $\theta \vdash \theta'$ then $\theta'\in \mu$). An ideal $\mu$ is said to be $S1$ on a set of parameters $A$ if it is $A$-invariant and if $\varphi(x,a_{i})\wedge \varphi(x,a_{j})\in \mu$ for $i\centernot= j < \omega$ implies $\varphi(x,a_{i})\in \mu$ for any $A$-indiscernible sequence $(a_{i})_{i}$ and any formula $\varphi(x,y)$. Under the assumption of the existence of an $S1$ ideal of formulas in a group, he defines a notion of stabilizer and shows that it is a type-definable subgroup.

\vspace{10pt}
Using the two notions of stabilizers, the Hrushovski one and the Kim-stabilizer one we prove under certain assumptions on a NSOP$_1$ field $F$ and a group $G$ definable in $F$ that there exists a typical-definable subgroup of $G$ of bounded index and a finite kernel embedding of this subgroup into the $F$-rational points of an algebraic group definable over $F$. We also show that the groups definable in $\omega$-free PAC fields of characteristic $0$ satisfy these assumptions.

\section{NSOP$_1$ groups}

\subsection{Kim-generics}

In this section we consider the notion of generic from stability and simplicity theory and try to see if we can find an analogy in the context of NSOP$_1$. The answer to this question seems to be negative: we give here some examples of `bad' behavior of NSOP$_1$ groups. The main example we consider in this section is the model completion of the theory of abelian groups with a generic unary function $f$ such that $f(0)=0$. We denote this theory by $T_{f,ab}$.

\vspace{10pt}
This theory is constructed in \cite{bossut2024some}. It is an NSOP$_1$ theory with existence and has quantifier elimination in the language $L= \lbrace +,0,f \rbrace$. Let $\mathbb{M}$ be a monster model of $T_{f,ab}$, we recall some of the properties of $\mathbb{M}$:
\begin{enumerate}
\item \cite[Remark 3.2.5]{bossut2024some} : The algebraic closure of any set $A\subseteq \mathbb{M}$ is equal to the structure generated by $A$.
\item \cite[Corollary 3.2.10]{bossut2024some} : Kim-forking independence coincides with algebraic independence $\forkindep^{a}$.
\item \cite[Proposition 3.2.9]{bossut2024some} : There is an independence relation $\forkindep^{\Gamma}$ that is invariant, stationary, satisfies existence, extension, transitivity, base monotonicity and symmetry.
\item \cite[Definition 3.3.1(4)]{bossut2024some} : For any subset $A\subseteq \mathbb{M}$ there is a type $p_{A}\in S(A)$ that satisfies the following property: Let $a\models p$ and let $\psi: A \rightarrow B$ be morphism of $L$-structures from $A$ to $B\subseteq \mathbb{M}$ and let $b\in B$ be any element. Then there is a unique extension of $\psi$ into a morphism from $ \langle A,a\rangle$ to $B$ sending $a$ to $b$. $p_{A}$ is called the $\Gamma$-generic type over $A$.
\end{enumerate}

\begin{remark} For any subsets $A\subseteq B \subseteq \mathbb{M}$ the type $p_{B}$ is the unique non-$\Gamma$ forking extension of $p_{A}$ to $B$, in particular $p_{A}$ is the  non-$\Gamma$ forking extension of $p_{\emptyset}$ to $A$.
\end{remark}

We now give the `naive' definition of a Kim-generic element of a group.

\begin{definition}
    Let $(G,\cdot)$ be a group definable over $E$ in a model $M$ of an NSOP$_1$ theory with existence. For $E\subseteq B$ a set of parameters we say that an element $a\in G$ is \emph{Kim-generic in $G$ over $B$} if $a\cdot b\forkindep_{B}^{K}b$ for every $b\in G$ such that $a\forkindep^{K}_{B}b$.
\end{definition}

\begin{remark}
The group $(\mathbb{M},+)$ does not have Kim-generics: let $a\in \mathbb{M}$, and consider $b$ such that $a\forkindep^{a}_{\langle 0 \rangle}b$ with the function $f$ being equal to $0$ everywhere in $\langle a,b\rangle\setminus \langle a \rangle$ except that $f(a+b)=b$. Then $a\forkindep^{a}_{\langle 0 \rangle}b$ and $a+b\centernot\forkindep^{a}_{\langle 0 \rangle}b$, so $a$ is not Kim-generic. This behavior was already observed in the theory $sT^{ACF}_{\infty}$ in \cite[Proposition 8.14]{dobrowolski2020sets}. A similar counter-example can be constructed with a model as the basis (i.e. instead of $\langle 0 \rangle$ here).\end{remark}

\begin{prop}\label{typeglobinv} There is a global type $p$ in $\mathbb{M}$ which is invariant under translation and does not $\Gamma$-fork over $\langle 0 \rangle$, meaning that if $g\models p_{\vert A}$ for a set $A$ and if $a\in A$ then $g + a\models p_{\vert A}$ and $g + a\forkindep^{\Gamma}_{\langle 0 \rangle}A$.
\end{prop}

\begin{proof}
Consider first the $\Gamma$-generic type $p_{\langle 0 \rangle}$ over $\langle 0 \rangle$. We define $p$ as its global non-$\Gamma$-forking extension $p=p_{\mathbb{M}}$. By definition this type does not $\Gamma$-fork over $\langle 0 \rangle$. We now show that it is invariant under translations. Let $A$ be a set of parameters, $a\in A$ and $g\models p_{\vert A}$.

\vspace{10pt}
By base monotonicity $g + a \forkindep^{\Gamma}_{a}A$, so by transitivity and stationarity of $\forkindep^{\Gamma}$ it is enough to show that $g + a\forkindep^{\Gamma}_{\langle 0 \rangle}a$ and that $g + a \equiv_{\langle 0 \rangle} g$. For the second point by quantifier elimination it is enough to show that there is a morphism $\varphi: \langle g + a \rangle \rightarrow \langle g \rangle$ sending $g + a$ to $g$: In fact then the morphism sending $g$ to $g + a$ (which exists by the $\Gamma$-genericity of $g$) is its inverse so these two structures are isomorphic.

\vspace{10pt}
We know that $g\forkindep^{\Gamma}_{\langle 0 \rangle}a$. By genericity of $g$ there is a morphism sending $g$ to $g+a \in \langle g,a \rangle$. Then $\forkindep^{\Gamma}$-independence implies that there is a morphism $\psi: \langle a, g+a \rangle  \rightarrow \langle a,g \rangle$ that sends $a$ to $a$ and $g+a$ to $g$. We can consider the restriction of $\psi$ to $\langle g + a\rangle$ which is the morphism we want.

\begin{figure}[hbtp]
    \centering

\begin{tikzcd}
                                                       &  & {\langle a,g \rangle}                                                                       &  &                                                         \\
                                                       &  &                                                                                             &  &                                                         \\
                                                       &  & {\langle a,g+a \rangle} \arrow[uu, "\psi" description, dashed]                              &  &                                                         \\
\langle a \rangle \arrow[rru] \arrow[rruuu, bend left] &  &                                                                                             &  & \langle g \rangle \arrow[llu] \arrow[lluuu, bend right] \\
                                                       &  & \langle 0 \rangle \arrow[llu, "\subseteq" description] \arrow[rru, "\subseteq" description] &  &                                                        
\end{tikzcd}

\caption{The $\forkindep^{\Gamma}$-generic}

\end{figure}

Now we show that $g + a\forkindep^{\Gamma}_{\langle 0 \rangle}a$. For this we consider two morphisms $\psi_{1}: \langle g + a\rangle \rightarrow D$ and $\psi_{2}: \langle a\rangle \rightarrow D$ and we show that they factorize through $\langle g+a,a\rangle =\langle g,a\rangle$. Consider the morphism $\psi_{0}: \langle g\rangle \rightarrow D$ sending $g$ to $\psi_{1}(g + a) - \psi_{2}(-a)$, since $g\forkindep^{\Gamma}_{\langle 0 \rangle}a$ there is a morphism $\psi: D \rightarrow \langle g,a\rangle$ factorizing $\psi_{0}$ and $\psi_{1}$ through $D$, and it is easy to see that it also factorizes $\psi_{2}$.
\end{proof}

\begin{lemma} Let $T$ be any theory with a definable group $G$. Consider a formula $\varphi(x,b)$ that forks over a set of parameters $A$. Let $g\forkindep^{f}_{A}b$, then the formula $\varphi(g\cdot x,b)$ forks over $A$.
\end{lemma}

\begin{proof} If $\varphi(g\cdot x,b)$ does not fork over $A$ there is some $a\models \varphi(g\cdot x,b)$ such that $a\forkindep^{f}_{A}g,b$. Then $g,a\forkindep^{f}_{A,g}g,b$, and since $g\forkindep_{A}b$ by transitivity we have $g\cdot a\forkindep^{f}_{A}b$, contradicting that $\varphi(x,b)$ forks over $A$.
\end{proof}

\begin{remark} This does not hold for Kim-forking, even under stronger conditions: Let $b\in \mathbb{M}$, consider the formula $\varphi (x,b):=f(x)=b$. This formula Kim-divides over $\langle 0 \rangle$. Let $g\models p_{\vert b}$, with $p$ the type from \cref{typeglobinv}. The formula $\varphi (g+ x,b)$ Kim-divides over $g$, however it does not Kim-divide over $\langle 0 \rangle$.

\vspace{10pt}
To make things simpler assume that $f$ is the identity on $\langle b \rangle$. Consider $a$ such that $a \centernot\in \langle b,g\rangle$, $a+a =0$, and $f$ is the identity on the group generated by $a,\langle b,g \rangle$ with the exception of $f(g+a)=b$ (this defines in fact a structure in $K$). Then $a\forkindep^{K}_{\langle 0 \rangle}g,b$ and $\models \varphi (g+a,b)$ (also $a\forkindep^{K}_{b}g$, $g\forkindep_{\langle 0 \rangle}^{\Gamma}b$ and $a\centernot\forkindep^{K}_{g}b$).
\end{remark}

\vspace{10pt}
To sum up in an NSOP$_1$ group, translating a Kim-forking or non-Kim-forking formula by some independent parameter can either give a Kim-forking or a non-Kim-forking formula.

\subsection{The Kim-Stabilizer}

\subsubsection{Strengthenings of the Independence Theorem in NSOP$_1$}

We know that NSOP$_1$ theories with existence satisfy 3-amalgamation for Lascar strong types. We will write $a\equiv^{L}_{E}b$ to mean that $a$ and $b$ have the same Lascar strong type over $E$. 

\begin{theorem}\label{indepth}\cite[Theorem 5.6]{dobrowolski2022independence} Let $b\forkindep^{K}_{E}c$ and $a\equiv^{L}_{E}a'$ with $a\forkindep^{K}_{E}b$ and $a'\forkindep_{E}c$. Then $\tp(a/Eb)\cup \tp(a'/Ec)$ does not Kim-fork over $E$.
\end{theorem}

In this subsection we prove some strengthenings of this result that we will use in the following subsection. For this we begin by recalling some facts about Kim-independence in NSOP$_1$ theories. It is shown in \cite[Theorem 2.5.10]{bossut2022kimforking} that \cref{indepth} holds in hyperimaginaries, this with the fact that $a\forkindep^{K}_{E}b$ if and only if $bdd(Ea)\forkindep^{K}_{bdd(E)}bdd(Eb)$ for all tuples $a,b,e\in \mathbb{M}^{heq}$ (\cite[Remark 2.2.13]{bossut2022kimforking} and symmetry of $\forkindep^{K}$) implies that we can assume that the set of parameter $E$ we consider in the amalgamation problem is boundedly closed.

\begin{definition} Let $E\subseteq \mathbb{M}$ be a set of parameters and $p(x)$ be a type over $E$. We call $p$ a \emph{Kim-amalgamation base} if whenever $p_{1}$ and $p_{2}$ are non-Kim-forking extensions of $p$ over some sets of parameters $A$ and $B$ respectively, with $A,B\subseteq \mathbb{M}^{eq}$ such that $A\forkindep^{K}_{E}B$, then the type $p_{1}\cup p_{2}$ does not Kim-fork over $E$. Any complete type over a model or more generally a boundedly closed set is a Kim-amalgamation base: \cite[Theorem 2.5.8]{bossut2022kimforking}.\end{definition}

\begin{prop}\cite[Proposition 2.2]{kim2021weak}\label{witnessing} Let $a_{0}\forkindep^{K}_{E}B$ with $E\subseteq B$ and let $I=(a_{i})_{i<\omega}$ be a Kim-Morley sequence over $E$. Then there is $I'\equiv_{Aa_{0}}I$ such that $I'$ is Kim-Morley over $B$.
\end{prop}

\begin{theorem}\cite[Theorem 2.8]{chernikov2023transitivitylownessranksnsop1}\label{witnessing2} Let $a_{0}$ be a tuple and let $E\subseteq B$ be sets of parameters. Then $a_{0}\forkindep^{K}_{E}B$ if and only if there exist $I=(a_{i})_{i<\omega}$ a Kim-Morley sequence over $E$ which is $B$-indiscernible.
\end{theorem}

Using these results Kim shows the following version of the independence theorem:

\begin{theorem}\cite[Theorem 2.7]{kim2021weak}\label{kimindtheorem} Let $b\forkindep^{K}_{E}c$ and $a\equiv^{L}_{E}a'$ with $a\forkindep^{K}_{E}b$ and $a'\forkindep_{E}c$. Then there is $\alpha \models \tp(a/Eb)\cup \tp(a'/Ec)$ such that $\alpha\forkindep^{K}_{E}bc$, $\alpha\forkindep^{K}_{Eb}bc$ and $\alpha\forkindep^{K}_{Ec}bc$.
\end{theorem}

We begin with a strengthening of this result, which was already known, and we give a proof for the sake of completeness.

\begin{prop}\label{thamalgsup} An NSOP$_1$ theory with existence satisfies the following version of the independence theorem: Whenever there is $E,b,c,\beta,\gamma \in \mathbb{M}$ such that $\beta \forkindep^{K}_{E}b$, $\gamma \forkindep^{K}_{E}c$, $\beta\equiv^{L}_{E}\gamma$ and $b\forkindep^{K}_{E}c$ there exists an $\alpha$ such that $\alpha\equiv_{Eb}\beta$, $\alpha\equiv_{Ec}\gamma$, $\alpha\forkindep^{K}_{Eb}bc$, $\alpha\forkindep^{K}_{Ec}bc$, $b\forkindep^{K}_{E\alpha}c$ and $\alpha\forkindep^{K}_{E}bc$.
\end{prop}

\begin{proof} We begin by showing that we can replace $b$ by a sequence $\Tilde{B}=(\Tilde{b}_{i})_{i<\kappa_{0}}$ which is $E$-Morley and Kim-Morley over $E\beta$ and $Ec$. Consider an $E$-Morley sequence $B'=(b'_{i})_{i<\kappa}$ such that $b'_{0}=b$ for $\kappa$ a large enough cardinal.

\vspace{10pt}
Since $\alpha\forkindep^{K}_{E}b'_{0}$ by \cref{witnessing} we can assume that $\beta\forkindep^{K}_{E}B'$ and that $B'$ is Kim-Morley over $E\beta$. Similarly since $c\forkindep^{K}_{E}b'_{0}$ there is $B''\equiv_{Ec}B'$ such that $B''\forkindep^{K}_{E}c$ and that $B''$ is Kim-Morley over $Ec$. By applying the independence theorem to $\tp(B'/E\beta)$ and $\tp(B''/Ec)$ over $E$ we get a new sequence $\Tilde{B}=(\Tilde{b}_{i})_{i<\kappa}$ which is $E$-Morley and Kim-Morley both over $E\beta$ and $Ec$. Since $\Tilde{B}\equiv_{Ec}B$ by moving $\beta$ we can assume that $\Tilde{b}_{0}=b$, which concludes the first step.

\vspace{10pt}
Now we apply \cref{kimindtheorem} to $\tp(\beta/E\Tilde{B})$ and $\tp(\gamma/Ec)$. We find some $\alpha' \models \tp(\beta/E\Tilde{B})\cup\tp(\gamma/Ec)$ such that $\alpha'\forkindep^{K}_{E}\Tilde{B}c$, $\alpha'\forkindep^{K}_{E\Tilde{B}}\Tilde{B}c$ and $\alpha'\forkindep^{K}_{Ec}\Tilde{B}c$. By Erdős-Rado we can find a sequence $\Tilde{B}'=(\Tilde{b}'_{i})_{i<\omega}$ which is $E\alpha' c$-indiscernible and finitely based on $\Tilde{B}$ over $E\alpha' c$. Since $\Tilde{B}$ is $Ec$-Kim-Morley and $E\alpha'$-Kim-Morley this implies that $\Tilde{B}'\equiv_{Ec}\tilde{B}$ and $\Tilde{B}'\equiv_{E\alpha'}\tilde{B}$. So we can find a sequence $B=(b_{i})_{i<\omega}$ with $b_{0}=b$ and $\alpha$ such that $\alpha B\equiv_{Ec}\alpha' \Tilde{B}'$.

\vspace{10pt}
$B$ is then an $E$-Morley sequence which is also Kim-Morley over $E\alpha$ and $Ec$. In particular the sequence $(b_{i})_{1\leq i <\omega}$ is $Eb_{0}$-Morley and $Eb_{0}c$-indiscernible, by \cref{witnessing2} this implies that $B\forkindep^{K}_{Eb_{0}}c$. Since $\alpha \forkindep^{K}_{EB}c$ transitivity yields that $\alpha\forkindep^{K}_{Eb}c$. We also have $\alpha\forkindep^{K}_{Ec}cb$. Since $B$ is Kim-Morley over $E\alpha$ and $E\alpha c$-indiscernible by \cref{witnessing2} $b\forkindep^{K}_{E\alpha}c$, so $\alpha$ is the tuple we want.\end{proof}

\begin{definition}\label{kimamalgaholds} We say that \emph{Kim-amalgamation holds} between two sets of parameters $B_{0}$ and $B_{1}$ over a common subset $E$ if for all $a^{0},a^{1}$ such that $a^{0}\equiv^{L}_{E}a^{1}$, $a^{0}\forkindep^{K}_{E}B_{0}$ and $a^{1}\forkindep^{K}_{E}B_{1}$, then there is an $a$ such that $a\equiv_{B_{i}}a^{i}$ for $=0,1$ and $a\forkindep^{K}_{E}B_{0}B_{1}$.
\end{definition}

We know that Kim-amalgamation over $E$ holds between two sets whenever they are Kim-independent over $E$, that is just the usual independence theorem for Kim-forking. However in some NSOP$_1$ theories the condition that Kim-amalgamation between some sets of parameters $B_{0}$ and $B_{1}$ holds over $E$ is weaker than the condition $B_{0}\forkindep^{K}_{E}B_{1}$. That is the case in $\omega$-free PAC fields for example with \cref{obois}. The point of the following lemma is to show that under this weaker assumption we can still get the stronger conclusion of \cref{kimindtheorem}.

\begin{lemma}\label{amalgweaker} Assume that $B_{0}$ and $B_{1}$ are two sets containing $E$ such that Kim-amalgamation holds between $B_{0}$ and $B_{1}$ over $E$. Then, for all $a^{0},a^{1}$ such that $a^{0}\equiv^{L}_{E}a^{1}$, $a^{0}\forkindep^{K}_{E}B_{0}$ and $a^{1}\forkindep^{K}_{E}B_{1}$, then there is an $a$ such that $a\equiv_{B_{i}}a^{i}$ for $=0,1$ and $a\forkindep^{K}_{E}B_{0}B_{1}$ and additionally $a\forkindep^{K}_{EB_{0}}B_{0}B_{1}$ and $a\forkindep^{K}_{EB_{1}}B_{0}B_{1}$.
\end{lemma}

\begin{proof} This proof is just a restatement of the one of \cite[Theorem 2.7]{kim2021weak} where we make sure that we only use the independence theorem over the parameters $B_{0}$ and $B_{1}$. Consider an $E$-Morley sequence $I_{0} = (a^{0}_{i})_{i<\kappa}$ with $a^{0}_{0}=a^{0}$ and $\kappa$ a large enough cardinal. Since $a^{0}\forkindep_{E}B_{0}$ by \cref{witnessing} we can assume that $I_{0}\forkindep^{K}_{E}B_{0}$ and that $I_{0}$ is Kim-Morley over $B_{0}$.

\vspace{10pt}
Similarly we can find an $E$-Morley sequence $I_{1}=(a^{i}_{1})_{i<\kappa}$ such that $a^{i}_{1}=a^{1}$, $I_{1}\equiv_{E}I_{0}$, $I_{1}\forkindep^{K}_{E}B_{1}$ and $I_{1}$ is Kim-Morley over $B_{1}$. By assumption there is a sequence $I'=(a_{i})_{i<\kappa}$ such that $I'\equiv_{B_{i}}I_{i}$ for $i=0,1$ and $I'\forkindep^{K}_{E}B_{0}B_{1}$. By Erdős-Rado we can find a sequence $I=(a_{i})_{i < \omega}$ which is indiscernible over $B_{0}B_{1}$ and finitely based on $I'$ over $B_{0}B_{1}$. Let $a:=a_{0}$.

\vspace{10pt}
Then $I\forkindep^{K}_{E}B_{0}B_{1}$, $I$ is $E$-Morley, Kim-Morley over $B_{0}$ and over $B_{1}$, and $a\equiv_{B_{i}}a^{i}$ for $i=0,1$. By \cref{witnessing2} $a\forkindep^{K}_{B_{0}}B_{1}$ and $a\forkindep^{K}_{B_{1}}B_{0}$, so $a$ is the tuple we want.
\end{proof}

\subsubsection{Defining the Kim-stabilizer}

In this subsection we consider the notion of stabilizer: We give a definition extending the one of $f$-stabilizers in simple theories (see \cite[Chapter 4]{wagner2000simple} on this topic) to the NSOP$_1$ context. We begin by recalling the definition of the $f$-stabilizer.

\begin{definition}\cite[Definition 4.1.17]{wagner2000simple}\label{defstabsimple}
Let $(G,\cdot)$ be a group definable over $E$ in a model $M$ of a theory with existence and let $p\in S(E)$ be a type concentrating in $G$. We define the following set:

\begin{center}
$S^{f}(p):=\lbrace g$ : $\exists a \models p,g\cdot a \models p, a\forkindep^{f}_{E}g$
and $g\cdot a \forkindep^{f}_{E}g\rbrace$.\end{center}

For $a\in G$ we will write $S^{f}(a/E)$ to mean $S^{f}(\tp(a/E))$. We denote by $Stab^{f}(p)$ the subgroup of $G$ generated by $S^{f}(p)$, which we call the \emph{$f$-stabilizer} of $p$.
\end{definition}

\begin{prop}\cite[Proposition 4.1.21]{wagner2000simple}\label{cassimple} Let $(G,\cdot)$ be a type-definable group in a simple theory and $p$  be a type concentrating in $G$. If $p$ is Lascar-strong then $Stab^{f}(p)=S^{f}(p)\cdot S^{f}(p)$ is a type-definable subgroup of $G$. 
\end{prop}

\begin{definition}
Let $(G,\cdot)$ be a group definable over $E$ in a model $M$ of an NSOP$_1$ theory with existence and let $p\in S(E)$ be a type concentrating in $G$. We define the following set:

\begin{center}

$S^{K}(p):=\lbrace g$ : $\exists a \models p,g\cdot a \models p, a\forkindep^{K}_{E}g$ and $g\cdot a \forkindep^{K}_{E}g\rbrace$.
\end{center}

For some $a\in G$ we will write $S^{K}(a/E)$ to mean $S^{K}(\tp(a/E))$. We denote by $Stab^{K}(p)$ the subgroup of $G$ generated by $S^{K}(p)$, which we call the \emph{Kim-stabilizer} of $p$.
\end{definition}

If $q\in S(E)$ is a non-Kim-forking extension of $p$ then $S^{K}(q)\subseteq S^{K}(p)$. The set $S^{K}(p)$ is type-definable for a given complete type $p$, and it coincides with $S^{f}(p)$ in simple theories.

\vspace{10pt}
In fact if $a\models p$ then $S^{K}(p)$ is defined by the partial type:
\begin{center}
    
$\exists y ( p(y) \cup p(x\cdot y)\cup \lbrace \neg \varphi(x,y)$ : $\varphi(x,a)$ Kim-forks over $E \rbrace $

\hspace{100pt}$\cup \lbrace \neg \varphi(x,x\cdot y)$ : $\varphi(x,a)$ Kim-forks over $E \rbrace )$
\end{center}

\begin{lemma}\label{stabilizateur} Let $(G,\cdot)$ be a group type-definable over $E$ in a model $M$ of an NSOP$_1$ theory with existence and let $p\in S(E)$ be a Kim-amalgamation base concentrating in $G$. Consider $g,g'\in S^{K}(p)$ such that Kim-amalgamation holds between $g$ and $g'$ over $E$. Then $g\cdot g' \in S^{K}(p)$.
\end{lemma}

\begin{proof}
Let $g,g'\in S^{K}(p)$ as in the statement. By definition there are $a,a'\models p$ such that $g\cdot a, g'\cdot a'\models p$, $g\forkindep^{K}_{E}a$, $g'\forkindep^{K}_{E}a'$ and $g\cdot a\forkindep^{K}_{E}a$, $g'\cdot a'\forkindep^{K}_{E}a'$. By assumption we can apply \cref{amalgweaker} to $\tp(a/g)$ and $\tp(g'\cdot a'/g')$ to find some $a''\models \tp(a/g)\cup \tp(g'\cdot a'/g')$ such that $a''\forkindep^{K}_{Eg}g'$ and $a''\forkindep^{K}_{Eg'}g$. 

\vspace{10pt}
Then $g'^{-1}\cdot a''$ is the witness we want. In fact:
\begin{center}
$g'^{-1}\cdot a''\models p$ and $g\cdot g' \cdot (g'^{-1}\cdot a'' )\models p$. Also $a''\forkindep^{K}_{Eg}g,g'$ which implies $g\cdot g'\cdot a''\forkindep^{K}_{Eg}g,g'$.
\end{center}

Since $g\cdot a \forkindep^{K}_{E}g$ transitivity yields $g\cdot g'\cdot (g'^{-1} \cdot a'') \forkindep^{K}_{E}g\cdot g'$. Similarly $a''\forkindep^{K}_{Eg'}g,g'$ implies $g'^{-1}\cdot a''\forkindep^{K}_{Eg'}g,g'$. Since $a' \forkindep^{K}_{E}g'$ by transitivity $g'^{-1}\cdot a''\forkindep^{K}_{E}g\cdot g'$ holds, so $g\cdot g'\in S^{K}(p)$.
\end{proof}

\begin{remark} In particular $S^{K}(p)$ is stable under Kim-independent product: If $g,g'\in S^{K}(p)$ and $g\forkindep^{K}_{E}g'$ then $g\cdot g'\in S^{K}(p)$. In simple theory we can show that if a type-definable subset $X$ of a group is stable by $\forkindep^{f}$-independent product then the subgroup generated by $X$ is $X\cdot X$. It is not clear that we can generalize this result to NSOP$_1$ since the proof revolves around the notion of stratified $\Delta$-rank, which does not generalize well to NSOP$_1$. 
\end{remark}

\begin{remark}\label{goulougoulou} $S^{K}(p)$ is also stable by inversion: If $g\in S^{K}(p)$ and $a\models p$ is such that $g\cdot a\models p$, $g\forkindep^{K}_{E}a$ and $g\cdot a\forkindep^{K}_{E}a$. Let $a':=g\cdot a$, then $a', g^{-1} \cdot a'\models p$, $g^{-1}\forkindep^{K}_{E}a'$ and $g^{-1} \cdot a'\forkindep^{K}_{E}a'$, so $g^{-1}\in S^{K}(p)$.
\end{remark}

\subsubsection{Examples}

We give some examples of groups $G$ definable in strictly NSOP$_1$ theories and types $p$ such that $Stab^{K}(p):=S^{K}(p)\cdot S^{K}(p)$.

\vspace{10pt}
\textbf{1} - The group $G=(\mathbb{M},+)$ in $T_{f,ab}$, where $T_{f,ab}$ is the model completion of the theory of abelian groups with a generic unary function $f$ such that $f(0)=0$.

\begin{lemma} For any unary type $p\in S(E)$ over some set $E$ of parameters $Stab^{K}(p)$ is a type-definable subgroup of $G$ and is equal to $S^{K}(p)+S^{K}(p)$. \end{lemma}

\begin{proof}
We define the following independence relation: say that two structures $A,B$ extending $E$ are $0$-independent, written $A\forkindep^{0}_{E}B$ if $A\forkindep^{a}_{E}B$ and $f(a+b)=0$ for all elements $a\in A\setminus E$, $b\in B\setminus E$. It is easy to show that $\forkindep^{0}$ satisfies existence, extension, invariance, symmetry, stationarity, transitivity.

\vspace{10pt}
Let $E\subseteq \mathbb{M}$ be a substructure and $p\in S(E)$ be a type. Consider $g_{1},g_{2},g_{3}\in S^{K}(p)$. Consider $g\equiv_{E}g_{1}$ such that $g\forkindep^{0}_{E}g_{1}g_{2}g_{3}$. Then any term in $g+g_{2}$ over $E$ is equal to $n\cdot (g+g_{2}) + e$ for some $n\in \mathbb{Z}$ and $e\in E$. From this it is clear that $g+g_{2}\forkindep^{K}_{E}g_{3}$ since $g\forkindep^{K}_{E}g_{2},g_{3}$. Now $g_{1} - g \in S^{K}(p)$ and $g + g_{2} + g_{3} \in S^{K}(p)$ by \cref{stabilizateur}. So $g_{1}+g_{2}+g_{3} = (g_{1} - g) + (g + g_{2} + g_{3}) \in S^{K}(p)+S^{K}(p)$ and the group generated by $S^{K}(p)$ is the type-definable group $S^{K}(p)+S^{K}(p)$.\end{proof}


\vspace{10pt}
\textbf{2} - Let $T_{F}$ be the theory of an NSOP$_1$ field $\mathbb{F}$ with existence such that $S^{K}(q)+ S^{K}(q)=Stab^{K}(q)$ for every type $q\in (F^{\kappa},+)$ where $\kappa$ is a small cardinal. By \cref{cassimple} this is true when $T_{F}$ is a simple theory. We consider the additive group $(\mathbb{V},+)$ of a monster model $(\mathbb{V},\mathbb{F})$ of the theory $sT^{T_{F}}_{\infty}$ of vector space of infinite dimension over a model of $T_{F}$ with a non-degenerate symmetric bilinear form.

\vspace{10pt}
We will also assume that $T$ satisfies existence, which holds when $T_{F}$ is a simple theory, as we have shown in \cite[Proposition 3.10]{bossut2023note}.

\vspace{10pt}
For all $E\subseteq A,B$ algebraically closed $A\forkindep^{K}_{E}B$ if and only if $V(A)$ and $V(B)$ are linearly independent over $V(E)$ and $F(A)\forkindep^{F}_{F(E)}F(B)$, where $\forkindep^{F}$ is Kim-forking independence in the sense of the theory of the field $T_{F}$.

\begin{lemma}\label{kimstabbilinear}
Consider a unary type $p(x)\in S(E)$ in the additive group $(\mathbb{V},+)$, where $E$ is an algebraically closed set of parameters. Then the equality $Stab^{K}(p):=S^{K}(p)\cdot S^{K}(p)$ holds.\end{lemma}

\begin{proof}
Fix $(e_{i})_{i<\kappa}$ a $K(E)$ basis of $V(E)$. If $p$ specifies that $x\in \langle E\rangle$, let $v\models p$. Then $v=\sum\limits_{i<\kappa} \lambda_{i} \cdot e_{i}$. Consider the type $q((y_{i})_{i<\kappa}):=\tp((\lambda)_{i<\kappa}/K(E))$. It is easy to see from the characterization of Kim-forking that $g\in S^{K}(p)$ if and only if $g = \sum\limits_{i<\kappa} \lambda'_{i} \cdot e_{i}$ and $(\lambda'_{i})_{i<\kappa}\in S^{K}(q)$. From this we deduce that $g \in S^{K}(p)$ if and only if $g = \sum\limits_{i<\kappa} \lambda'_{i} \cdot e_{i}$ and $(\lambda'_{i})_{i<\kappa}\in S^{K}(q)$, so $Stab^{K}(p)=S^{K}(p)\cdot S^{K}(p)$ since $Stab^{K}(q)=S^{K}(q)\cdot S^{K}(q)$ by assumption.

\vspace{10pt}
In the other case $p$ specifies that $x\centernot\in \langle E\rangle$. Let $v\models p$ and consider the type $$q((y_{i})_{i<\kappa}):=\tp(([v,e_{i}])_{i<\kappa}/K(E)).$$

\vspace{10pt}
\textbf{Claim:} $u\in S^{K}(p)$ if and only if $([u,e_{i}])_{i<\kappa}\in S^{K}(q)$ for every $u\in \mathbb{V}$.

\vspace{10pt}
\textit{Proof:} The inclusion $\subseteq$ is clear, for the other direction consider $u\in \mathbb{V}$ such that $([u,e_{i}])_{i<\kappa}\in S^{K}(q)$. There is $(\lambda_{i})_{i<\kappa}\models q$ such that $(\lambda_{i}+[u,e_{i}])_{i<\kappa}\equiv_{K(E)}(\lambda_{i})_{i<\kappa}$, $(\lambda_{i})_{i<\kappa}\forkindep^{F}_{E}([u,e_{i}])_{i<\kappa}$ and $(\lambda_{i}+[u,e_{i}])_{i<\kappa}\forkindep^{F}_{E}(\lambda_{i})_{i<\kappa}$. By extension we can assume that $(\lambda_{i})_{i<\kappa}\forkindep^{F}_{E([u,e_{i}])_{i<\kappa}}K(acl(E,u))$, so $(\lambda_{i})_{i<\kappa}\forkindep^{F}_{E}K(acl(E,u))$ by transitivity.

\vspace{10pt}
Let $w'\models p$. Then $([w',e_{i}])_{i<\kappa}\equiv_{K(E)}(\lambda_{i})_{i<\kappa}$, so there is $\alpha, \beta \in \mathbb{K}$ such that: $$([w',e_{i}])_{i<\kappa}[w',w']\equiv_{K(E)}(\lambda_{i})_{i<\kappa}\alpha \equiv_{K(E)}(\lambda_{i}+[u,e_{i}])_{i<\kappa}\beta .$$ Similarly since $(\lambda_{i})_{i<\kappa}\forkindep^{F}_{E}K(acl(E,u))$ and $K(acl(E,u))\forkindep^{F}_{E}(\lambda_{i}+[u,e_{i}])_{i<\kappa}$ by extension we can assume that $(\lambda_{i})_{i<\kappa}\alpha\forkindep^{F}_{E}K(acl(E,u))$ and $K(acl(E,u))\forkindep^{F}_{E}(\lambda_{i}+[u,e_{i}])_{i<\kappa}\beta $.

\vspace{10pt}
By \cite[Lemma 2.1.1]{bossut2023note} we can find $w\in \mathbb{V}$ such that:

\begin{enumerate}
    \item $w\centernot\in \langle Eu \rangle$,
    \item $[w,e_{i}]=\lambda_{i}$ for every $i<\kappa$,
    \item $[w,w]= \alpha$,
    \item $[w,u]= \frac{1}{2}\cdot ( [u,u]-\alpha - \beta )$.
\end{enumerate}

Then $[w+u,w+u]=\beta$. By quantifier elimination and the characterization of Kim-forking it is clear that $u\forkindep^{K}_{E}w$, $u\forkindep^{K}_{E}w+u$ and $w+u\equiv_{E}w\equiv_{E}w'$, so $u\in S^{K}(p)$.\hspace*{0pt}\hfill\qedsymbol{}

\vspace{10pt}
As in the case of $p\in \langle E \rangle$ we deduce by the assumption on $Stab^{K}(q)$ in the field that $Stab^{K}(p)=S^{K}(p)\cdot S^{K}(p)$.
\end{proof}

A similar proof of the \textbf{Claim} would work in the context of the additive group $(\mathbb{V}^{n},+)$ and the semi-direct product $\mathbb{V}^{n}\rtimes GL_{n}(\mathbb{K)}$.
\vspace{10pt}

\textbf{3} - Type-definable groups of finite arity in $\omega$-free PAC fields of characteristic $0$. We will need the following results:

\begin{theorem}\label{caractkimforkomegafree} Let $E=acl(E)$, $E \subseteq A,B$ for $i=0,1$ be algebraically closed subsets of an $\omega$-free PAC field $F$ of characteristic $0$. Then $A\forkindep^{K}_{E}B$ if and only if $A$ and $B$ are linearly disjoint over $E$ and $A^{acl}B^{acl} \cap F=AB$.
\end{theorem}

\begin{proof} Kaplan and Ramsey prove in \cite[Theorem 9.33]{kaplan2020kim} that if $E$ is a model Kim-independence over $E$ is equivalent to weak independence over $E$, a notion defined in \cite[Definition 1.2]{Chatzidakis2002PropertiesOF} By Chatzidakis. Chatzidakis proves in \cite[Lemma 2.8]{Chatzidakis2002PropertiesOF} that in an $\omega$-free PAC-field $F$ if $E\subseteq A,B$ are algebraically closed subfields of $F$ then $A$ and $B$ are weakly independent  over $E$ if and only if  $A$ and $B$ are linearly disjoint over $E$ and $A^{acl}B^{acl} \cap F=AB$. We can extend the result of Kaplan and Ramsey to $E$ algebraically closed by using the Kim-Pillay characterization of Kim-forking over arbitrary sets \cite[Theorem 5.1]{chernikov2023transitivitylownessranksnsop1}.
\end{proof}

\begin{lemma}\cite[Lemma 3.2]{Chatzidakis2002PropertiesOF}\label{obois} Let $E=acl(E)$, $E \subseteq A_{i},B_{i}$ for $i=0,1$ be algebraically closed subsets of an $\omega$-free PAC field $F$ such that $A_{i}\forkindep^{K}_{E}B_{i}$ for $i=0,1$, $A_{0}\equiv_{E}A_{1}$ and $B_{0}\forkindep^{ACF}_{E}B_{1}$. Then there is $A\equiv_{B_{i}}A_{i}$ for $i=0,1$ such that $A\forkindep^{K}_{E}B_{0}B_{1}$.
\end{lemma}

\begin{remark}\label{prodgeneric}\cref{obois} together with \cref{amalgweaker} implies that if $p$ is a type in a type-definable group of finite arity then $S^{K}(p)$ is stable by $\forkindep^{ACF}$ independent product. The following proof uses this and the fact that $\forkindep^{ACF}$ has generic elements in $G$, namely the elements of maximal transcendence degree over $E$.

\end{remark}
\begin{lemma}\label{stabomegafree} Consider a Lascar strong type $p\in S(E)$ in a group $(G,\cdot)$ which is finitary and type-definable over $E$ in a $\omega$-free PAC field $F\supseteq E$. Then $Stab^{K}(p):=S^{K}(p)\cdot S^{K}(p)$.
\end{lemma}

\begin{proof}
Notice that if the transcendence degree of the type $p$ over $E$ is $n<\omega$ then the elements of $S^{K}(p)$ have transcendence degree $\leq n$ over $E$. Consider $g_{i} \in S^{K}(p)$ for $i=0,1,2$. Consider $g\in S^{K}(p)$ of maximal transcendence degree over $E$ such that $g\forkindep^{K}_{E}g_{0},g_{1},g_{2}$. Since $g\forkindep^{K}_{E}g_{0}$ we know that $g_{0}\cdot g^{-1}\in S^{K}(p)$ by \cref{stabilizateur} ($g^{-1}\in S^{K}(p)$ since $S^{K}(p)$ is closed by inversion - \cref{goulougoulou}). We want to show that $g\cdot g_{1}\cdot g_{2}\in S^{K}(p)$.

\vspace{10pt}
By \cref{stabilizateur} we know that $g\cdot g_{1} \in S^{K}(p)$. We show that $g\cdot g_{1}\forkindep^{ACF}_{E}g_{2}$, where $\forkindep^{ACF}_{E}$ is forking independence in the sense of the algebraic closure. This follows from the fact that $g\cdot g_{1}$ and $g$ are interalgebraic over $Eg_{1}g_{2}$ and that $g\forkindep^{ACF}_{E}g_{1}g_{2}$: $\trdeg(g\cdot g_{1}/E) \geq \trdeg(g\cdot g_{1}/Eg_{1}g_{2})=\trdeg(g/Eg_{1}g_{2})=\trdeg(g/E)$ is maximal in $S^{K}(p)$, so $\trdeg(g\cdot g_{1}/E)=\trdeg(g\cdot g_{1}/Eg_{2})$ and $g\cdot g_{1}\forkindep^{ACF}_{E}g_{2}$.

\vspace{10pt}
Then \cref{prodgeneric} yields that $g\cdot g_{1}\cdot g_{2}\in S^{K}(p)$, so $g_{0}\cdot g_{1}\cdot g_{2}\in S^{K}(p) \cdot S^{K}(p)$ and $Stab^{K}(p) = S^{K}(p)\cdot S^{K}(p)$.
\end{proof}

\begin{remark} The previous proof relies on the finiteness of the transcendence degree over $E$ to show that $g\cdot g_{1}\forkindep^{ACF}_{E}g_{2}$. In the more general case of a group of infinite arity this argument would not work.
\end{remark}

\subsubsection{A Counter-Example}

We now show that the Kim-stabilizer does not satisfy certain properties of the $f$-stabilizer in simple theories. For this we consider the example of the additive group $(\mathbb{V},+)$ in $sT^{ACF}_{\infty}$.

\vspace{10pt}
Let $(e_{i})_{i<\omega}$ be a family of linearly independent vectors. For $n<\omega$ let $p_{n}\in S(e_{<n})$ be the type defined by $x\centernot \in \langle e_{<n} \rangle$, $[x,x]=0$ and $[x,e_{i}]=0$ for all $i<n$. $p_{m}$ is a non-Kim-forking extension of $p_{n}$ for all $n<m$. From the proof of \cref{kimstabbilinear} it is easy to see that: 
\begin{center}

$Stab^{K}(p_{n}) = \lbrace x$ : $[x,e_{i}]=0$ for all $i<n \rbrace$.
\end{center}

We can then notice two things happening in this case:

\begin{enumerate}
    \item[1 -] $Stab^{K}(p_{n})$ is not a subgroup of bounded index of $Stab^{K}(p_{m})$ when $n<m$.
    \item[2 -] There is $v\in Stab^{K}(p_{m})$ such that $v\forkindep^{K}_{e_{<n}}e_{<m}$ and such that $v + Stab^{K}(p_{m})$ seen as an hyperimaginary is not in the bounded closure of $e_{<m}$. In fact it is easy to see that $v + Stab^{K}(p_{m}) \in bdd(A)$ for some set $A$ if and only if $[v,e_{i}]\in bdd(A)$ for all $i<m$.
\end{enumerate}

These two things can not happen in a simple group. After seeing this we are left with the following questions. Let $G$ be a group type-definable over $E$ in an NSOP$_1$ theory with existence: 

\begin{enumerate}
    \item[1 -] If $p\in S_{G}(A)$ and $E\subseteq A \subseteq B$ is there a non-Kim-forking extension $q\in S_{G}(B)$ of $p$ such that $Stab^{K}(q)$ has bounded index in $Stab^{K}(p)$?
    \item[2 -] Is there a type $p\in S_{G}(E)$ such that $Stab^{K}(p)$ has bounded index in $G$?
\end{enumerate}

$1$ holds in the group $(\mathbb{V},+)$ in $sT^{ACF}_{\infty}$: Given any type $p\in S_{\mathbb{V}}(A)$ if $q$ is the $\forkindep^{\Gamma}$-independent extension of $p$ to $B$ then $Stab^{K}(p)=Stab^{K}(q)$.

\section{Stabilizer and expansion of the language}

In this subsection we define a set of conditions on a field $F$ and on a group $G$ type-definable in $F$. We show that under these conditions there is a definable finite to one homomorphism from a type-definable subgroup of $G$ of bounded index to an algebraic group $H$.

\subsection{S1 ideals of formulas}

We begin by recalling the notions of $S1$ ideal and wide types from \cite{hrushovski2011stablegrouptheoryapproximate}.

\begin{definition} Let $X$ be a definable set over some set of parameters $A$. Let $\mu$ be an $A$-invariant ideal of definable subsets of $X$. $\mu$ is said to be $S1$ if $\varphi(x,a_{0})\wedge \varphi(x,a_{1})\in \mu$ if and only if $\varphi(x,a_{0})\in \mu$ for every formula $\varphi (x,y)$ and $A$-indiscernible sequence $(a_{i})_{i<\omega}$.
\end{definition}

\begin{definition} Let $\mu$ be an ideal. A definable set $X$ is medium if $\mu$ is $S1$ when restricted to $X$.
\end{definition}

\begin{definition} Let $\mu$ be an ideal. A partial type $\pi(x)$ is said to be $\mu$-wide (or just wide if there is no ambiguity) if it does not imply a formula in $\mu$. A type is medium if it concentrates on a medium set.
\end{definition}

\begin{lemma}\cite[Lemma 2.9]{hrushovski2011stablegrouptheoryapproximate} \label{wideimpiesnonfork} Let $\mu$ be an $A$-invariant ideal which is $S1$ on some $A$-definable set $X$. Then any wide type $p$ concentrating in $X$ does not fork over $A$.
\end{lemma}

\begin{remark}\label{boundedindex} If $G$ is a type-definable group over $A$, $\mu$ is an ideal $S1$ on $G$,  $A$ invariant and invariant under translations. Then any wide subgroup $H\leq G$ type-definable over $A$ has bounded index in $G$.
\end{remark}

\begin{proof} If $H\leq G$ does not have bounded index there is an $A$-indiscernible sequence $(a_{i})_{i<\omega}$ such that all of the $a_{i}$ are in different cosets of $H$. Let $\pi(x)$ be a partial type over $A$ defining $H$. Then $\pi(a_{i}^{-1}\cdot x)\wedge \pi(a_{j}^{-1}\cdot x)$ is inconsistent for every $i\centernot=j <\omega$.

\vspace{10pt}
By indiscernibility and compactness there is a formula $\varphi(x)\in \pi(x)$ such that $\varphi(a_{i}^{-1}\cdot x)\wedge \varphi(a_{j}^{-1}\cdot x)$ is inconsistent. So $\varphi(a_{0}^{-1}\cdot x)\wedge \varphi(a_{1}^{-1}\cdot x)\in \mu$, and $\varphi(a_{0}^{-1}\cdot x)\in \mu$ since $\mu$ is $S1$. Since $\mu$ is invariant under translation $\varphi(x)\in \mu$ which contradicts the fact that $H$ is wide.
\end{proof}

In this context Hrushovski defines the following notion of stabilizer:

\begin{definition} Let $G$ be a group type-definable over $A$. Let $\mu$ be an $A$-invariant ideal in $G$ which is invariant under translations by elements of $G$. If $p\in G$ is a wide type we define:
\begin{center}
$S^{\mu}(p):=\lbrace g$ : $p \cdot g \cap p$ is wide$\rbrace$. 
\end{center}

We define $Stab^{\mu}(p)$ as the subgroup of $G$ generated by $S^{\mu}(p)$.
\end{definition}

\begin{definition} Let $p$ and $q$ be two types over some model $M$. We define:
\begin{center}
    $p\times_{nf} q:= \lbrace (a,b)$ : $a\models p,b\models q, b\forkindep^{f}_{M}a \rbrace$
\end{center}
\end{definition}

We will use the two following results:

\begin{theorem} \cite[Theorem 2.15]{montenegro2018stabilizersgroupsfgenericsntp2}\label{stabextended} Let $\mu$ and $\lambda$ be $M$ invariant ideals of $G$ that are invariant under left and right translation and such that $\mu$ is $S1$ in any $X\in \lambda$. Assume that we are given a wide and medium type $p$ in $G$ and that the following conditions are satisfied:

\begin{enumerate}
    \item[(A)] for any types $q,r$, if for some $(c,d)\models q \times_{nf} r$, $\tp(cd/M)$ or $\tp(dc/M)$ is $\lambda$-medium then $q$ is $\lambda$-medium;
    \item[(B)] for any $(a,b)\models p\times_{nf}p$, $\tp(a^{-1}b/M)$ is $\lambda$-wide.
    \item[(F)] there are $(a,b)\models p\times_{nf}p$ such that $\tp(a/Mb)$ does not fork over $M$.
\end{enumerate}

Then $Stab^{\mu}(p)=S^{\mu}(p)^{2}=(pp^{-1})^{2}$ is a type-definable, wide and $\lambda$-medium group. Also $Stab^{\mu}(p)\setminus S^{\mu}(p)$ is contained in a union of non-wide $M$-definable sets.
\end{theorem}

\begin{lemma}\cite{hrushovski2011stablegrouptheoryapproximate}[Lemma 2.16]\label{newmod} Let $\mu$ be an $A$-invariant ideal. Then there is a model $M \supseteq A$ and a global $M$-invariant type $p$ finitely satisfiable in $M$ such that if $a\models p_{\vert M}$ and $b\models p_{\vert Ma}$ then $\tp(a/Mb)$ is wide.
\end{lemma}

\subsection{The assumptions on $F$}

Let $\mathbb{F}$ be a monster model of some theory $T$ in a countable language $L$ extending the theory of fields and let $(G,\cdot)$ be a group of finite arity type-definable over a set of parameters $A\subseteq \mathbb{F}$. We write $\mathbb{F}^{alg}$ the algebraic closure of $\mathbb{F}$ in the sense of fields and similarly $B^{alg}\subseteq F^{alg}$ the algebraic closure of some $B\subseteq \mathbb{F}$. We will write $dcl^{alg}$ for the definable closure in the sense of $\mathbb{F}^{alg}$. We will write $\forkindep^{ACF}$ for forking-independence in the sense of $\mathbb{F}^{alg}\models ACF$.

\vspace{10pt}
We assume the following:

\begin{enumerate}
    \item[(H1)] $T$ is NSOP$_1$, satisfies existence and $Stab^{K}(p)=S^{K}(p)\cdot S^{K}(p)$ for every type $p\in S_{G}(B)$ and parameters $B \supseteq A$.
    \item[(H2)] The group law of $G$ is algebraic in the sense of $\mathbb{F}^{alg}$: $a^{-1}\in (A,a)^{alg}$ and $a\cdot b\in (A,a,b)^{alg}$ for every $a,b\in G$.
    \item[(H3)] We can extend the language $L$ to $L'$ and $\mathbb{F}$ in an $L'$-structure which we write $\mathbb{F}'$ such that there is an invariant $S1$ ideal of formulas $\mu_{0}$ which is  invariant by translation in $\mathbb{F}'$ and for which $G$ is wide.
\end{enumerate}

If $N'\prec \mathbb{F}'$ we will write $N$ for $N'$ seen as an $L$-structure, in this context we have that $N\prec \mathbb{F}$. We recall the following result about Kim-independence in fields:

\begin{prop}\cite[Proposition 9.28]{kaplan2020kim}\label{kimindcorpsquelc} Suppose $F$ is an arbitrary field. We denote by $A^{s}$ the separable closure a field $A$. Let $M\prec F$ and $a,b\in F$ such that $a\forkindep^{K}_{M}b$. Then the fields $A = acl(Ma)$ and $B = acl(Mb)$ satisfy the following conditions:
\begin{enumerate}
    \item $A$ and $B$ are linearly disjoint over $M$.
    \item $F$ is a separable extension of $AB$.
    \item $acl(AB)\cap A^{s}B^{s}=AB$.
\end{enumerate}
\end{prop}

With this we can show the following in NSOP$_1$ fields:

\begin{lemma}\label{h4} Let $F$ be an NSOP$_1$ field with existence. Let $E\subseteq A,B \subseteq F$ be algebraically closed sets such that $A\forkindep^{K}_{E}B$. Then $A\forkindep^{ACF}_{E}B$.
\end{lemma}

\begin{proof} Let $(A_{i})_{i<\omega +\omega +1}$ be an $E$-Morley sequence with $A_{0}=A$ which is $B$-indiscernible and satisfies 
$$(A_{i})_{i<\omega +\omega +1}\forkindep^{K}_{E}B.$$

Let $I=(A_{i})_{i<\omega}$ and $J=(A_{i})_{\omega \leq i < \omega+\omega}$. Since $\tp(A_{\omega +\omega}/EIJB)$ is finitely satisfiable in $J$ we have that $A_{\omega+\omega}\forkindep^{ACF}_{J}EIJB$. By stability and quantifier elimination in $ACF$ the sequence $(A_{i})_{i<\omega +\omega +1}$ is totally indiscernible in the theory $ACF$, so $A_{\omega+\omega}\forkindep^{ACF}_{I}EIJB$. Let $\overline{a}:=\cb(A_{\omega +\omega}/EIJB)$, then $\overline{a}\in I^{alg}\cap J^{alg}$ by elimination of imaginaries in ACF.

\vspace{10pt}
We show that $\overline{a}\in E^{alg}$, which implies that $A_{\omega+\omega}\forkindep^{ACF}_{E}IJB$, so $A\forkindep^{ACF}_{E}B$ by indiscernibility. We know that $\overline{a}\in I^{alg}\cap J^{alg}$.

\vspace{10pt}
Let $P(X)$ be the minimal polynomial of an element $a\in \overline{a}$ over the field generated by $I$. Then $P$ has coefficients in the field generated by $A_{<k}$ for some $k<\omega$. Let $Q$ be the minimal polynomial of $a$ over the field generated by $A_{[k,2k[}$. Since $A_{<k}$ and $A_{[k,2k[}$ begin a $JA_{\omega+\omega}$-indiscernible sequence in $F^{alg}$ we know that $A_{<k}\equiv_{J^{alg}A_{\omega+\omega}}A_{[k,2k[}$.

\vspace{10pt}
In particular $P$ and $Q$ have the same degree. If $P\centernot= Q$ then $a$ is a root of $P-Q$ which has degree strictly smaller than $P$ and coefficients in the field generated by $I$, which contradicts the fact that $P$ is the minimal polynomial of $a$ over $I$. So $P=Q$ and $P$ has coefficients in the intersection of the fields generated by $A_{<k}$ and $A_{[k,2k[}$, which is $E$ since $(A_{i})_{i<\omega}$ is a total $E$-Morley sequence. From this we have that $a\in E^{alg}$ for any $a\in \overline{a}$, and so that $\overline{a}\in E^{alg}$.
\end{proof}

\begin{lemma}\label{nonkimforkwide2} If a type $p'=\tp_{L'}(a/N'B)$ is wide for an elementary substructure $N'\prec \mathbb{F}'$, then $p:=\tp_{L}(a/NB)$ its restriction to $L$ does not Kim-fork over $N$ (here $N$ is the restriction of $N'$ to the language $L$).
\end{lemma}

\begin{proof}
Let $(B_{i})_{i<\omega}$ a coheir Morley sequence with $B_{0}=B$ over $N'$. Then since $\tp(a'/N'B_{0})$ does not fork in the sense of $L'$ there is an $a''\in F$ such that $a''NB_{i}\equiv_{N}a'NB_{0}$ in the sense of $L'$ for all $i<\omega$, so in particular $a''B_{i}\equiv_{N}aB_{0}$ for all $i<\omega$ in the sense of $L$. Since $(B_{i})_{i<\omega}$ is coheir Morley over $N$ in the sense of $L$, by Kim's Lemma for Kim-forking, we get that $a'\forkindep^{K}_{N}B$.\end{proof}


\begin{lemma}\label{conditionf2} There exist a model $M'$ and a global $M'$-invariant type $p'\in S_{G}(\mathbb{F})$ in the language $L'$ finitely satisfiable in $M'$ such that $A \subseteq M' \prec  \mathbb{F}'$, if $b\models p'_{\vert M'}$ and $a\models p'_{\vert M'b}$ then $\tp(b/M'a)$ is wide and also that $M'$ has infinite transcendence degree over $A$. We write $M$ for $M'$ seen as an $L$-structure (so $M\prec \mathbb{F}$) and $p\in S_{G}( \mathbb{F})$ for the restriction of $p'$ to $L$.\end{lemma}

\begin{proof}
Let $A\subseteq B$ be such that $B$ contains a sequence $(g_{i})_{i<\omega}\in G$ of elements of maximal transcendence degree over $A$ such that $g_{i}\forkindep^{ACF}_{A}g_{<i}$ for all $i<\omega$. We apply \cref{newmod} to $B$ and find a model $A \subseteq B \subseteq M' \prec  \mathbb{F}'$ and a global $M'$-invariant type $p'$ finitely satisfiable in $M'$ such that if $a\models p'_{\vert M'}$ and $b\models p'_{\vert M'a}$ then $\tp(a/M'b)$ is wide. Then $M'$ has infinite transcendence degree over $A$.\end{proof}

In this section we are going to work on three different levels:
\begin{enumerate}
    \item The stable theory of $\mathbb{F}^{alg}$, i.e. the quantifier free types, in which we will build a group configuration in the following subsection.
    \item The original $L$-theory of the field $\mathbb{F}$ which is NSOP$_1$ and has a notion of stabilizer.
    \item The $L'$-theory of $\mathbb{F}'$ which has an $S1$ ideal, so we can apply the results of \cite{montenegro2018stabilizersgroupsfgenericsntp2}.
\end{enumerate}

\subsection{Definable measures and (H3)}

\begin{definition} A measure $\mu$ is said to be definable over a set of parameters $B$ if for every formula $\varphi(x,a)$ with parameters $a$, whenever $\mu(\varphi(x,a))\in ]\alpha,\beta[$ there is a formula $\theta(y)\in \tp(a/B)$ such that $\mu(\varphi(x,a))\in ]\alpha,\beta[$ for all $\models \theta (a')$.
\end{definition}

A measure definable over $B$ is $B$-invariant (i.e. invariant under automorphisms fixing $B$). If $\varphi(x,y)$ is a formula and $]\alpha,\beta[$ is an open interval then what the definition states is that the set $\lbrace b $ : $\mu(\varphi(x,b))\in ]\alpha,\beta[\rbrace$ is $\vee$-definable over $B$.

\vspace{10pt}
By expanding the language it is always possible to extend a measure $\mu$ on the definable subsets of some definable set $X$ into a definable measure on the definable subsets of $X$ in this expansion.

\vspace{10pt}
For this all we have to do is to add to the language the relation symbols $R_{\varphi,q}$ for $\varphi(x,y)$ a formula such that $\varphi(x,b)\subseteq X$ for all $b$ and $q\in \mathbb{Q}$ which are interpreted by 
\begin{center}
    $\lbrace b $ : $\mu(\varphi(x,b))<q\rbrace$.
\end{center}

Then extend the measure to the new definable sets, which is possible by applying \cite[Lemma 7.3]{simon2015guide}, and iterate the process. Notice that if the measure $\mu$ is defined on the definable subsets of a definable group $G$ and is invariant by translation there is no reason for the expanded measure that we get in this process to also be invariant by translation.

\vspace{10pt}
In the context of the assumption $(H3)$ this is a problem: We want to get an ideal which is invariant by translation and also by automorphisms, and if we have a measure $\mu$ on the definable subsets of $G$ in $\mathbb{F}$ we would like this ideal to correspond to the ideal of formulas of measure $0$ in the expanded language.

\vspace{10pt}
With the following two lemmas, we show that in a very specific case, it is indeed possible to expand the language and extend a measure which is invariant by translation into a definable measure that remains invariant by translation, so the ideal of formulas of measure $0$ in the expansion will be invariant by translation and also $S1$, and with it we can satisfy $(H3)$.

\begin{lemma}\label{mesurefinie} Consider a definable measure $\mu$ on definable subsets of some $B$-definable set $X$ for some set of parameters $B$ such that for every formula $\varphi(x,y)$ the measure $\mu(\varphi(x,b))$ can only take a finite number of values for $b$ ranging over a model. Then for every formula $\varphi(x,y)$ and real number $\alpha$ there is a formula $\theta^{\varphi}_{\alpha}(y)$ such that $\models \theta_{\varphi,\alpha}(a)$ if and only if $\mu(\varphi(x,a)) = \alpha $ for all $y$-tuples $a$.
\end{lemma}

\begin{proof} Let $\varphi(x,y)$ be a formula. The set $\lbrace b $ : $\mu(\varphi(x,b))\in ]\alpha,\beta[\rbrace$ is $\vee$-definable over $A$ for any open interval. In the case when the measure can only take a finite number of values this set coincides with the pre-image of a closed interval, so it is both open and closed in $S_{y}(A)$, i.e. it is a definable set.
\end{proof}

\begin{lemma}\label{mesureprod} Let $L$ be a countable language. Let $(M_{e})_{e<\omega}$ be a family of $L$-structures and $G_{e}\subseteq M_{e}$ be a uniformly definable family of groups. Assume that $\mu_{e}$ is a measure on definable subsets of $G_{e}$ which is definable over a set of parameters $A_{e}\subseteq M_{e}$, can only take a finite number of values for a given formula $\varphi(x,y)$ and is invariant by translation. 
Assume also that $\mu_{e}(X_{e})=1$ for all $e<\omega$. Let $\mathcal{U}$ be an ultrafilter on $\omega$. Let $T$ be the $L$-theory of the ultraproduct $M:=\prod_{i}M_{e}/\mathcal{U}$. 

\vspace{10pt}
Then there is an expansion of $L$ into a countable language $L'$ and an $L'$-structure on $M$ such that there exist a measure $\mu'$ on $G:= \prod_{e}G_{e}/\mathcal{U} \subseteq M$ which extends the ultraproduct measure $\prod_{e}\mu_{e}/\mathcal{U}$, is definable on $\prod_{e}A_{e}/ \mathcal{U}$ in the sense of $L'$ and is invariant by translation.
\end{lemma}

\begin{proof}
We expand the language $L$ into a language $L'$ which formulas are interpreted in $M$ as ultraproduct of $L$-formulas in the coordinates $M_{e}$.

\vspace{10pt}
This is done through an inductive construction. Let $L_{0}=L$. Assume that $L_{i}$ has been defined. For every $L_{i}$-formula $\varphi(x,y) = \prod_{e}\varphi^{e}(x,y)/\mathcal{U}$ which implies $x\in X$, integers $n$ and $1\leq j \leq n$ we add a new relation symbol $\theta_{\varphi,j,n}(y)$ that is interpreted in $M$ as the ultraproduct of the formulas $\theta^{e}_{\varphi^{e},j,n}(y)$, where $\theta^{e}_{\varphi^{e},j,n}(y)$ expresses that $\frac{j-1}{n} < \mu_{e}(\varphi^{e}(x,y))< \frac{j+1}{n}$ in the coordinates $M_{e}$ (such a formula exists by \cref{mesurefinie}).

\vspace{10pt}
These new formulas are measurable in a canonical way, if $\theta(x) = \prod_{e<\omega}\theta^{e}(x)/\mathcal{U}$ is an $L_{i+1}$ formula that implies that $x\in G$, then $\theta^{e}(x)$ is an $L$-formula which implies that $x\in G_{e}$ for almost all $e$ and we can take the measure $\mu'(\theta(x))$ to be the standard part of $\prod_{e<\omega}\mu_{e}(\theta^{e}(x))/\mathcal{U}$. Since all of the measures $\mu_{e}$ are invariant by translation we have that the standard part of $\prod_{e<\omega}\mu_{e}(\theta^{e}(x))/\mathcal{U}$ is equal to the standard part of $\prod_{e<\omega}\mu_{e}(\theta^{e}(g_{e}\cdot x))/\mathcal{U}$ for every $g = [g_{e}]_{e<\omega} \in G$, so $\mu'(\theta^{e}(x))=\mu'(\theta^{e}(g\cdot x))$ and the measure we get is invariant by translation.

\vspace{10pt}
By construction if $\varphi(x,y)$ is a formula in $L_{i}$, there is $\theta_{\varphi,i,n}(y) \in L_{i+1}$ such that $M \models \theta_{\varphi,i,n}(a)$ if and only if $\mu(\varphi(x,a)) \in ] \frac{j-1}{n}, \frac{j+1}{n}[$ for all $y$-tuple $a$ in $M$. So the measure $\mu'$ is definable in $L'$. Since the language $L$ is countable $L_{i}$ is also countable for every $i<\omega$, so $L'= \bigcup\limits_{i<\omega} L_{i}$ also is.\end{proof}

In the context of the previous lemma we have a definable group which is definably amenable by some measure $\mu$ in the language $L$, and an expansion $L'$ of the language and $\mu'$ of the measure in which the group is definably amenable with a definable measure in the sense of $L'$. As we saw before it is not true in general that we can get a definable measure which is invariant by translation from a measure on definable subsets which is invariant by translation by extending the language, and it is an interesting question whether we can write some sufficient conditions for this to happen.

\vspace{10pt}
In technical terms the proof of \cref{mesureprod} consists of adding some ultraproducts of the definitions of the measures to the language, take the generated algebra, extend the measure to all of these sets (since the ultraproduct measure extends to the algebra of internal sets of the ultraproduct) and iterate this process.

\subsection{The group configuration}

In this subsection we construct a group configuration over $M'$ in the sense of $\mathbb{F}^{alg}$ with points in $\mathbb{F}$. This consists of a set of tuples $(a,b,c,x,y,z)$ that we can arrange as in \cref{figuregroupconf} so that:

\begin{enumerate}
    \item Two points in the same line are $\forkindep^{ACF}$-independent over $M'$.
    \item Two points in the same line are inter-algebraic over the third one and $M'$.
    \item Any other pair or triplets are $\forkindep^{ACF}$-independent.
\end{enumerate}

\begin{figure}[hbtp]
    \centering
    
\begin{tikzcd}
  &                                               & a \arrow[ldd, no head] \arrow[rdd, no head] &                                               &   \\
  &                                               &                                             &                                               &   \\
  & b \arrow[ldd, no head] \arrow[rrrdd, no head] &                                             & x \arrow[rdd, no head] \arrow[llldd, no head] &   \\
  &                                               & z                                           &                                               &   \\
c &                                               &                                             &                                               & y
\end{tikzcd}

\caption{The group configuration}\label{figuregroupconf}

\end{figure}

\vspace{10pt}
Let $M'$ and $p'$ be as in \cref{conditionf2}. By \cref{wideimpiesnonfork} $p'$ does not $\forkindep^{f}$-fork over $M'$ in the sense of $L'$ and by \cref{nonkimforkwide2} $p$ does not Kim-fork over $M$ in the sense of $L$. By \cref{h4} $p$ does not $\forkindep^{ACF}$-fork over $M'$. Let $b\models p'_{\vert M'}$, $x\models p'_{\vert M'b}$ and $a\models p'_{\vert M'bx}$. Then $\tp_{L'}(a/M'b)$ is wide and both $\tp_{L'}(a/M'b)$ and $\tp_{L'}(b/M'a)$ do not fork over $M'$. 

\vspace{10pt}
Let $c:=b\cdot a$, $y:=a\cdot x$ and $z=b\cdot a\cdot x$. Then $\tp(a/M')$ is wide and since $a\equiv_{M'}b\equiv_{M'}x$ these three tuples have the same transcendence degree $n$ over $M'$. The tuples $a,b,c,x,y,z$ form a group configuration over $M'$ in $\mathbb{F}^{alg}$:

\begin{enumerate}
    \item Since $p'$ is finitely satisfiable in $M'$ we have that $a\forkindep^{ACF}_{M'}b,x$ and $x\forkindep^{ACF}_{M'}b$, so $a\forkindep^{ACF}_{b}$ and $a\forkindep^{ACF}_{M'}x$. Since $a \forkindep^{ACF}_{M'b}x$ by base monotonicity we have $c\forkindep^{ACF}_{M'b}x$ and by transitivity we deduce $c\forkindep^{ACF}_{M'}x$. Similarly we can show that $b\forkindep^{ACF}_{M'}y$. The types $\tp(b/M'a)$ and $\tp(x/M'a)$ are wide. Since the ideal $\mu_{0}$ is invariant by translation the types $\tp(c/M'a)$ and $\tp(y/M'a)$ are also wide, so by \cref{h4} $c\forkindep_{M'}^{ACF}a$ and $y\forkindep_{M'}^{ACF}a$. The remaining relations can be deduced from these by the properties of $\forkindep^{ACF}$.
    \item Two points in the same line are inter-algebraic over the third one and $M'$ since the third point can be obtained by using the group law which is definable over $M'$.
    \item As written previously $a\forkindep^{ACF}_{M'}b,x$ and $x\forkindep^{ACF}_{M'}b$. From the previous relations, the fact that the group law is algebraic in the sense of $F^{alg}$ and the properties of $\forkindep^{ACF}$ we can deduce the remaining independence relations for $3$.
\end{enumerate}

\subsubsection{Germs and the Hrushovski-Weil theorem}

We recall the definition of the germ of a function and the `Group Chunk Theorem' which we will apply in $\mathbb{F}^{alg}$. A presentation of this result can be found in \cite{bouscaren1989group}. We will here refer here to the course of Bays on geometric stability theory \cite{bays2018geometric} which suits us better. To keep things simple we state these notion in the context of $\mathbb{F}^{alg} \models ACF$.

\begin{definition}
Let $p_{0}, q_{0} \in S(A)$ be stationary types. In our context by stability and elimination of imaginaries this just means types over an algebraically closed subfield of $\mathbb{F}^{alg}$. Let $p,q\in S(\mathbb{F}^{alg})$ be their respective global non-forking extensions. 

\vspace{10pt}
Let $f$ be a definable partial function and let $\varphi(x,y,m)$ be a formula defining $f$ over the parameters $m$. We say that $f$ is defined at $p$ if $p(x) \models \exists!y(\varphi(x,y,m))$, which we can also write $p(x) \models x \in \operatorname{dom}(f)$.

\vspace{10pt}
Two definable partial functions $f_1,f_2$ defined at $p$ are said to have the same germ at $p$ if $p(x) \models f_1(x) = f_2(x)$. This is a definable equivalence relation between definable partial functions, the class of $f$ is denoted $\tilde{f}$ and we call it the \emph{germ of $f$ at $p$}.

\vspace{10pt}
To a germ $\tilde{f}$ corresponds a unique function on $p$: If $f_1,f_2$ are two representatives of $\tilde{f}$ defined over parameters $m_{1}$ and $m_{2}$ respectively, then $f_{1}(a)=f_{2}(a)$ whenever $a\models p\vert_{Am_{1}m_{2}}$. We write $\tilde{f}: p \rightarrow q$ if when $f$ is a representative of $\tilde{f}$ defined over parameters $b$ then $f(a)\models q\vert_{Ab}$ whenever $a\models p\vert_{Ab}$.

\vspace{10pt}
$Aut(\mathbb{F})$ acts on germs: Let $\sigma \in Aut(\mathbb{M})$ and $\tilde{f}: p \rightarrow q$ a germ on $p$. If $f_1,f_2$ are two representatives of $\tilde{f}$ then $f_1^{\sigma},f_2^{\sigma}$ have the same germ at $p^{\sigma}$, which we write $\tilde{f}^{\sigma}: p^{\sigma} \rightarrow q^{\sigma}$.

\vspace{10pt}
We say that a tuple $\bar{b}$ in $\mathbb{F}^{alg}$ is a code for $\tilde{f}$ if:

\begin{center}
$\bar{b} = \bar{b}^\sigma$ if and only if $\tilde{f} = \tilde{f}^\sigma$ for all $\sigma \in \operatorname{Aut}(\mathbb{F}^{alg})$.\end{center}

We then define $\ulcorner \tilde{f} \urcorner:= \operatorname{dcl}^{\mathrm{eq}}(\bar{b})$, this is in fact well-defined since two codes are fixed by the same automorphisms, so they are interdefinable. Since $ACF$ is stable and has elimination of imaginaries such a code always exists: The type $p$ is definable over $A$. Let 
$\varphi(x,y,n)$ be a formula defining a representative $f$ of $\tilde{f}$. Then $f=f_{n}$ where $(f_{z})_{z}$ is a definable family of partial functions defined by $\varphi(x,y,z)$. We consider the relation:

\begin{center} $E_{f}(m,m'):= d_{p}x.(f_{m}(x)=f_{m'}(x))$
\end{center}

Then $E_{f}(m,m')$ if and only if $f_{m}$ and $f_{m'}$ have the same germ at $p$. If $\sigma \in Aut(\mathbb{M})$ then $f_{m}^{\sigma}=f_{m^{\sigma}}$. From this it is easy to see that the $E_{f}$-class of $n$ is a code for $\tilde{f}$, i.e. $\ulcorner\tilde{f}\urcorner = \operatorname{dcl}^{\mathrm{eq}}(n/ E_{f})$.

\vspace{10pt}
If $\tilde{f}:p \rightarrow q$ and $\tilde{g}: q \rightarrow s$ are two germs we can define their composition $\tilde{g}\circ \tilde{f}: p \rightarrow s$, if $f$ and $g$ are representatives defined over parameters $n$ and $m$ respectively then $g \circ f$ is a representative of $\tilde{g}\circ \tilde{f}$.

\vspace{10pt}
If a germ $\tilde{f}:p \rightarrow q $ defines an injective function, i.e. if for any representative $f$ defined over parameters $n$ there is $\theta \in p$ such that $\models ((\theta(x)\wedge\theta(x') )\implies((f(x)=f(x'))\implies (x=x')))$, then we can define the inverse germ $\tilde{f}^{-1}$ of $\tilde{f}$ which is the unique germ $\tilde{f}^{-1}: q \rightarrow p$ such that $\tilde{f}^{-1}\circ \tilde{f} = \tilde{id}$.
\end{definition}

\begin{definition}
If $p, q, s \in S(N)$ are stationary, a \emph{family $\tilde{f}_s$ of germs $p \rightarrow q$} is a family $\tilde{f}_s:= (\tilde{f}_b)_{b \models s}$ of germs at $p$ of an $N$-definable family $f_b$ of partial functions, which is such that $\tilde{f}_b: p \rightarrow q$ whenever $b \models s$. We say that :

\begin{itemize}
  \item The family is \emph{canonical} if $b$ is a code for $\tilde{f}_b$, for all $b \models s$.
  \item The family is \emph{generically transitive} if $\tilde{f}_b(x) \forkindep^{f}_{N} x$ for some (any) $b$ and $x$ such that $b \models s$ and $x \models p\vert_{Nb}$.
\end{itemize}
\end{definition}

\begin{lemma}\cite[Lemma 5.1]{bays2018geometric}\label{generictranscano} Let $(b,z,y)$ be a triplet of tuples and $B$ be a tuple of parameters satisfying $b\forkindep^{ACF}_{B}z$, $b\forkindep^{ACF}_{B}y$ and $y\in dcl(Bbz)$. Define $s:=\tp(b/B)$, $q:=\tp(y/B)$ and $p:=\tp(z/B)$. Let $f_{b}(z)=y$ be a formula witnessing $y\in dcl(bz)$. We can consider the family of germs $\tilde{f}_{s} : p \rightarrow q $. 

\begin{itemize}
    \item[(i)] $\tilde{f}_{s}$ can be taken invertible if and only if $z\in dcl(Bby)$;
    \item[(ii)] $\tilde{f}_{s}$ is generically transitive if and only if $z\forkindep^{ACF}_{B}y$;
    \item[(iii)] $\ulcorner \tilde{f}_{b}\urcorner =cb(zy/Bb)$. So $\tilde{f}_{s}$ is canonical if and only if $cb(zy/Bb)=dcl(b)$.
\end{itemize}
\end{lemma}

\begin{definition} An homogeneous space consists of a group $G$ together with a set $S$ on which $G$ acts transitively. We say that $(G,S)$ is an $A$ type-definable homogeneous space for a set of parameters $A$ whenever $G$ is an $A$ type-definable group, $S$ is an $A$ type-definable set and the action of $G$ on $S$ is also type-definable over $A$. In that case we say that $S$ is connected if and only if there exist a global type $s\in S$ such that $g\cdot b\models s\vert_{Ag}$ whenever $g\in G$ and $b\models s\vert_{Ag}$. This type is then called the generic type of $S$.
\end{definition}

\begin{prop} \cite[Lemma 5.2]{bays2018geometric}\label{groupchunk}(Hrushovski-Weil Group chunk Theorem) Let $p, s \in S(N)$ be stationary, and suppose that $\Tilde{f}_{s}$ is a generically transitive canonical family of invertible germs $p \longrightarrow p$. Suppose that $\Tilde{f}_{s}$ is closed under inverse and generic composition, meaning that for $b\models s$ there exists $b'\models s$ such that $\Tilde{f}_{b}^{-1}= \Tilde{f}_{b'}$, and for $b_{1},b_{2}\models s$, $b_{1}\forkindep^{f}_{\emptyset}b_{2}$ there exists $b_{3} \models s$ such that:

\begin{center}
$\Tilde{f}_{b_{1}} \circ \Tilde{f}_{b_{2}} = \Tilde{f}_{b_{3}}$, and $b_{i}\forkindep^{f}_{\emptyset}b_{3}$ for $i=1, 2$.\end{center}

Then there exists a connected faithful $\emptyset$ type-definable homogeneous space $(G,S)$, a definable embedding of $s$ into $G$ as its unique generic type, and a definable embedding of $p$ into $S$ as its unique generic type, such that the generic action of $s$ on $p$ agrees with the action $*$ of $G$ on $S$, i.e. $\Tilde{f}_{b}(a)=b * a$ for $b\models s$ and $a\models p$ such that $a\forkindep^{f}_{\emptyset}b$.\end{prop}

\subsubsection{Obtaining an algebraic group}

In the next proof we follow similar steps as in \cite[Theorem 6.1]{bays2018geometric}. We begin by constructing a new group configuration $(a',b',c',x',y',z')$ that satisfies the conditions of \cref{groupchunk}. This new configuration is obtained by applying some modifications to the original configuration $(a,b,c,x,y,z)$ from the previous section. These modifications consist of replacing some tuple by another one that is interalgebraic over $M'$ in the sense of $\mathbb{F}^{alg}$. 

\vspace{10pt}
To keep track of the modifications we did we will also consider a set of parameters $A \subseteq B \subseteq M$, finitely generated over $A$. $(a,b,c,x,y,z)$ is a group configuration over $M'$ whose interalgebraicity relations are over $A$ (they all come from the group law).

\vspace{10pt}
If $A\subseteq A'\subseteq M'$ is such that $(a,b,c,x,y,z)\forkindep^{ACF}_{A'}M$ then $(a,b,c,x,y,z)$ is a group configuration over $A'$. Let $A'$ be such a set, $A'$ can be taken finitely generated over $A$ (we just add to $A$ the coefficients of the algebraic relations between $(a,b,c,x,y,z)$ over $M'$), and let $B:=A'$. We construct a sequence of sets $B$ and a sequence of group configurations over $B$.

\begin{prop}\label{groupconfigdsf} There is an $M'$-definable algebraic group $H$ and dimension generic elements $a',b',c'\in H(\mathbb{F})$ such that $a'\cdot b'=c'$, $(M'a')^{alg}=(M'a)^{alg}$, $(M'b')^{alg}=(M'b)^{alg}$ and $(M'c')^{alg}=(M'c)^{alg}$. In particular $a$ and $a'$, $b$ and $b'$, $c$ and $c'$ are interalgebraic over $M'$ in the model theoretic sense.\end{prop}

\begin{proof}
Let $(a',b',c',x',y',z') := (a,b,c,x,y,z)$. We begin by adding some parameters to $B$ from $M'$ in order to make $b'$ and $y'$ interdefinable over $Bz'$ in the sense of $\mathbb{F}^{alg}$.

\begin{figure}[!ht]
\centering
\resizebox{0.4\textwidth}{!}{%
\begin{circuitikz}
\tikzstyle{every node}=[font=\Huge]

\node [font=\Huge] at (17.5,9.5) {a'};
\node [font=\Huge] at (14.5,5.5) {b'};
\node [font=\Huge] at (12,1.5) {c'};
\node [font=\Huge] at (20.5,5.5) {x'};
\node [font=\Huge] at (23,1.5) {y'};
\node [font=\Huge] at (17.5,3.5) {z'};
\node [font=\Huge] at (23,10.75) {a''};
\node [font=\Huge] at (8.25,1.5) {c''};
\node [font=\Huge] at (23.25,5.5) {x''};
\draw [short] (12.5,1.5) -- (17.5,9);
\draw [short] (17.5,9) -- (22.5,1.5);
\draw [short] (20,5.25) -- (12.5,1.5);
\draw [short] (15,5.25) -- (22.5,1.5);
\draw [ color={rgb,255:red,228; green,0; blue,0}, short] (22.5,1.5) -- (23,10.25);
\draw [ color={rgb,255:red,228; green,0; blue,0}, short] (23,10.25) -- (8.75,1.5);
\draw [ color={rgb,255:red,228; green,0; blue,0}, short] (8.75,1.5) -- (22.75,5.5);
\end{circuitikz}
}%
\caption{A copy of the configuration over the line $(b'z'y')$}
\end{figure}

\vspace{10pt}
Let $a''\in G(M')$ of maximal transcendence degree such that $a''\forkindep^{ACF}_{B}a'b'c'x'y'z'$ (such an element exists since $M'$ has infinite transcendence degree over $B$). Define $c'':=b'\cdot a'$ and $x'':=a''^{-1}\cdot y'$. Let $\overline{z}'$ be the tuple of the coefficients of the minimal polynomial of $z$ over $Ay'b'x''c''$, then $\overline{z}'\in \mathbb{F}$. We replace the set of parameters $B$ by $B=Aa''$ and replace: $y'$ by $y'x''$, $b'$ by $b'c''$, $z'$ by $\overline{z}'$, so that now $z'\in dcl^{alg}(By'b')$.

\vspace{10pt}
Let $b''\in G(M')$ of maximal transcendence degree such that $b''\forkindep^{ACF}_{B}a'b'c'x'y'z'$. Define $c'':=b''\cdot a' $ and $z'':= c''\cdot x'$. Let $\overline{x} \in \mathbb{F}$ be the tuple of the coefficients of the minimal polynomial of $x'$ over $Ba'y'c''z''$. We replace the set of parameters $B$ by $Bb''$ and replace: $a'$ by $a'c''$, $y'$ by $y'z''$, $x'$ by $\overline{x}$, so that now $x'\in dcl^{alg}(Ba'y')$.

\vspace{10pt}
Let $c''\in G(M')$ of maximal transcendence degree such that $c''\forkindep^{ACF}_{B}a'b'c'x'y'z'$. Define $a'':=c''\cdot b'^{-1}$ and $x'':=c'^{-1}\cdot z'$. Let $\overline{y} \in \mathbb{F}$ be the tuple of the coefficients of the minimal polynomial of $y'$ over $Bb'a''x''$. We replace the set of parameters $B$ by $Bc''$ and replace: $b'$ by $b'a''$, $z'$ by $z'x''$, $y'$ by $\overline{y}$, so that now $y'\in dcl^{alg}(Bb'z')$ and $z'\in dcl^{alg}(Bb'y')$. $\cb(y'z'/Bb')\in \mathbb{F}^{alg}$, and is interalgebraic with $b'$ over $B$ in $\mathbb{F}^{alg}$, it consists of the coefficients of the algebraic relations of $y'z'$ on $Bb'$, which are included in $\mathbb{F}$. 

\vspace{10pt}
At this step we have a group configuration in $\mathbb{F}^{alg}$ with points inside of $\mathbb{F}$ such that $y'\in dcl^{alg}(Bz'b')$, $z'\in dcl^{alg}(Bb'y')$, $b'$ is interdefinable with $\cb(z'y'/Bb')$ over $B$. This group configuration is interalgebraic with the original group configuration over $M'$. Notice that whenever an element of $\mathbb{F}^{alg}$ is definable over parameters in $\mathbb{F}$ then it lies in $\mathbb{F}$ ($\mathbb{F}$ is definably closed in $\mathbb{F}^{alg}$), in particular if an imaginary in $\mathbb{F}^{alg,eq}$ is definable over parameters in $\mathbb{F}$ by elimination of imaginaries in $\mathbb{F}^{alg}$ it has a representative in $\mathbb{F}$.

\vspace{10pt}
Let $p^{alg}=\tp(y'/B^{alg})$, $r^{alg}=\tp(b'/B^{alg})$ and $q^{alg}=\tp(z'/B^{alg})$ be types in the sense of the stable theory of $\mathbb{F}^{alg}$. Let $f_{r}$ be the generic family of invertible germs from $p^{alg}$ to $q^{alg}$ corresponding to $(b',z',y')$. To apply \cref{groupchunk} to $(b',y',z')$ we need to check whenever $b''\models r^{alg}_{\vert Bb'}$ and $dcl(d):=\lceil \Tilde{f}_{b''}^{-1} \cdot \Tilde{f}_{b'} \rceil$ then $d\forkindep^{ACF}_{B}b'$ and $d\forkindep^{ACF}_{B}b''$. This follows from the same verification as in $(II)$ of \cite[Theorem 6.1]{bays2018geometric}.

\vspace{10pt}
We can apply the group chunk theorem \cref{groupchunk} to get a connected faithful homogeneous space $(H,S)$ type-definable in $\mathbb{F}^{alg}$ over $B$ and a definable embedding of $p^{alg}$ which is the generic type of $S$ into $H$ as its unique generic type $s_{G}$ in the sense of $\mathbb{F}^{alg}$.

\vspace{10pt}
We do one more base change to get a group configuration of $(H,S)$. Let $b''\in G(M')$ of maximal transcendence degree such that $b''\forkindep^{ACF}_{B}a'b'c'x'y'z'$. Define $y''=z'\cdot b''^{-1}$. By the previous remark about imaginaries we can find $g\in M'$ such that $dcl^{eq}(g)=\lceil \Tilde{f}_{b''}^{-1} \cdot \Tilde{f}_{b'} \rceil$. Then $g\models s_{G}$, $y',y''\models p^{alg}$ and $y''=g\cdot y'$.

\vspace{10pt}
We add $b''$ to the parameters and interalgebraically replace $b'$ by $g$ and $z'$ by $g\cdot y'\models p^{alg}$. Let $c''\in G(M')$ of maximal transcendence degree such that $c''\forkindep^{ACF}_{B}a'b'c'x'y'z'$. Define $b''=c''\cdot a'^{-1}$ and $z'':=c''\cdot x'^{-1}$. We add $c''$ to the parameters and interalgebraically replace $a'$ by $b''\models s_{G}$ and $x'$ by $z''$. Let $h:=b'\cdot a'^{-1}$. Then $x'=a'\cdot y'$, $z'=b'\cdot y'$ so $z'=h'\cdot x'$. By \cref{generictranscano} $\cb(x'z'/Ba'b')=h$. Since $x'$ and $c'$ are interalgebraic over $B$ we get that $\cb(x'z'/Ba'b')^{alg}=\cb(x'z'/Bc')^{alg}=acl^{eq}(Bc')$, and we can interalgebraically replace $c'$ with $h$. Then $(a',b',c',x',y',z')$ is a group configuration of $(H,S)$.

\vspace{10pt}
In we look at the edge $(a',b',c')$ of the configuration we have three points of a type-definable group $H$ over $B$ in $\mathbb{F}^{alg}$. Such a group $H$ is definable in $ACF$ and is definably isomorphic to an algebraic group, so by adding to $B$ the parameters under which this bijection is defined, which we can find in $M'$, and composing by this bijection we can assume that $(a',b',c')$ lies in an algebraic group $H$ definable over $B$. Then each tuple $a,b$ and $c$ are interalgebraic in the sense of $\mathbb{F}^{alg}$ over $M'$ with their corresponding primed tuple.
\end{proof}

\begin{theorem}\label{resultatgroup} There exists an algebraic group $H$ definable over $M$, an $M$ type-definable (in the sense of $L$) subgroup $G_{0}$ of bounded index in $G$ and a definable finite to one homomorphism from $G_{0}$ to $H(\mathbb{F})$.
\end{theorem}

\begin{proof}
Consider the algebraic group $H$ and the interalgebraic elements $a$ and $a'$ from \cref{groupconfigdsf}.

\vspace{10pt}
We define the ideal $\mu$ on $G\times H$ as set of the definable subsets of $G\times H$ whose projection to $G$ is in $\mu_{0}$. We define the ideal $\lambda$ as the set of the definable subsets of $G\times H$ whose projections to $G$ and $H$ each have finite fibers. Let $\Tilde{p}'=\tp_{L'}(a,a'/M')$ and $p'=\tp_{L'}(a/M')$. $a$ and $a'$ are interalgebraic over $M'$ so $\Tilde{p}$ is wide (in the sense of $\mu$ since $\tp(a/M')$ is wide in the sense of $\mu_{0}$) and medium (since its coordinates are interalgebraic), and satisfies the conditions of \cref{stabextended} (the conditions $A$ and $B$ follows from the same verifications as in the proof of \cite[Theorem 2.19]{montenegro2018stabilizersgroupsfgenericsntp2}, and $F$ holds by the choice of $M'$).

\vspace{10pt}
We can apply \cref{stabextended} to get that $K':=Stab^{\mu}(\Tilde{p}')=S^{\mu}(\Tilde{p})\cdot S^{\mu}(\Tilde{p}) \leq G\times H$ is a connected wide subgroup type-definable in $L'$. Since $K'$ is wide, $\pi_{1}(K')$ is $\mu_{0}$-wide, so by \cref{boundedindex} $\pi_{1}(K')$ is a type-definable subgroup of bounded index in $G$.

\vspace{10pt}
Consider the original language $L$ on $\mathbb{F}$. Let $\Tilde{p}$ be the restriction of $\Tilde{p}'$ to $L$ and $p$ the restriction of $p'$ to $L$. Define $K\leq G\times H$ as $K= Stab^{K}(\Tilde{p}) = S^{K}(\Tilde{p})\cdot S^{K}(\Tilde{p})$. $K$ is a group type-definable over $M$ by (H1).

\vspace{10pt}
\textbf{Claim 1:} $K' \leqslant K$.

\vspace{10pt}
\textit{Proof of Claim 1.} We show that $S^{\mu}(\Tilde{p}')\subseteq S^{K}(\Tilde{p})$. If $g\in S^{\mu}(\Tilde{p}')$ there is $\alpha\models \Tilde{p}'$ such that $g\cdot \alpha \models \Tilde{p}'$ and $\tp(\alpha/M'g)$ is $\mu$-wide. So $\alpha\models \Tilde{p}$ and $g\cdot \alpha \models \Tilde{p}$. Since $\mu$ is invariant by translation $\tp(g\cdot \alpha/M'g)$ is $\mu$-wide. By \cref{nonkimforkwide2} we have that $g\cdot \alpha \forkindep^{K}_{M}g$ and $\alpha\forkindep^{K}_{M}g$, so $g\in S^{K}(\Tilde{p})$.\hspace*{0pt}\hfill\qedsymbol{}

\vspace{10pt}
We consider the two projections $\pi_{1}: K \rightarrow G$ and $\pi_{2}: K \rightarrow H$. $K' \leq K \leq G\times H$, so $\pi_{1}(K)$ has bounded index in $G$. Let $K_{1}:= Ker(\pi_{1})$.

\vspace{10pt}
\textbf{Claim 2:} Both projections from $K$ to $G$ and $H$ have finite fibers.
\vspace{10pt}

\textit{Proof of Claim 2.} We begin by showing that whenever $(g,\Tilde{g})\in S^{K}(\Tilde{p})$ then $g$ and $\Tilde{g}$ are interalgebraic over $M$ in the sense of $\mathbb{F}^{alg}$. Let $(g,\Tilde{g})\in S^{K}(\Tilde{p})$, there is a realization $(a,\Tilde{a})\models \Tilde{p}$ such that $g,\Tilde{g}\forkindep^{K}_{M}a,\Tilde{a}$ and $g\cdot a,\Tilde{g}\cdot \Tilde{a}\models \tilde{p}$, so $a$ and $\Tilde{a}$ are interalgebraic over $M'$ and $g\cdot a$ and $\Tilde{g}\cdot \Tilde{a}\models p$ are also interalgebraic over $M'$, both in the sense of $\mathbb{F}^{alg}$.

\vspace{10pt}
From this we deduce that $g$ and $\Tilde{g}$ are interalgebraic over $Ma$. Since $g,\Tilde{g}\forkindep^{K}_{M}a,\Tilde{a}$ by base monotonicity of $\forkindep^{ACF}$ we get that $\Tilde{g}\forkindep^{ACF}_{Mg}a$, so $\trdeg(\Tilde{g}/M,g)=\trdeg(\Tilde{g}/M,g,a)=0$, and similarly $\trdeg(g/M,\Tilde{g})=0$, so $g$ and $\Tilde{g}$ are interalgebraic over $M$ in the sense of $\mathbb{F}^{alg}$.

\vspace{10pt}
Assume that $(1_{G},h)\in K$ for some $h\in H$. We show that in that case $h \in M^{alg}\cap \mathbb{F} = M$. There is $(g,\Tilde{g})\in S^{K}(\Tilde{p})$ such that $(g^{-1},\Tilde{g}^{-1}\cdot h)\in S^{K}(\Tilde{p})$. Consider some $(g',\Tilde{g}')\in S^{K}(\Tilde{p})$ such that $(g',\Tilde{g}')\forkindep^{K}_{M}g,\Tilde{g},h$. Then by \cref{stabilizateur} $(g\cdot g',\Tilde{g}\cdot \Tilde{g}')\in S^{K}(\Tilde{p})$ and $(g'^{-1}\cdot g^{-1},\Tilde{g}'^{-1}\cdot \Tilde{g}^{-1}\cdot h)\in S^{K}(\Tilde{p})$.

\vspace{10pt}
We know from the previous point and \cref{stabilizateur} that $g\cdot g'$ and $\Tilde{g}\cdot \Tilde{g}'$ and $g'^{-1}\cdot g^{-1}$ and $\Tilde{g}'^{-1}\cdot \Tilde{g}^{-1}\cdot h$ are interalgebraic over $M$, so in particular $h\in acl(M,g'\cdot g)$. Also since $g'\forkindep^{ACF}_{M}g,h$ we have $g\cdot g'\forkindep^{ACF}_{M,g}h$ and by maximality of the transcendence degree $g'\cdot g \forkindep^{ACF}_{M}h$, which implies that $h\in M^{alg}\cap \mathbb{F}=M$. $Ker(\pi_{1})$ is a type-definable subgroup included in $acl(M)$, so by compactness the partial type defining this set implies an algebraic formula and so $Ker(\pi_{1})$ is finite. The proof for $\pi_{2}$ is similar.\hspace*{0pt}\hfill\qedsymbol{}

\vspace{10pt}
By \textbf{Claim 2} $K_{1}$ is finite. Note that $K_{1}$ is normal in $K$. $K_{1}$ is central in $K$ since $K$ is connected over $M$. Let $C\leq H$ be the centralizer of $\pi_{2}(K_{1})$ inside $H$. $C$ is an algebraic subgroup of $H$ and is definable in $L$ over $M'$. We can replace $H$ by $C/\pi_{2}(K_{1})$ which is also an algebraic group definable over $M'$, so we can assume that $K_{1}$ is trivial, so $K$ defines a finite to one homomorphism from $\pi_{1}(K)$ to $H$ and $G_{0}=\pi_{1}(K)$ is the subgroup we want.\end{proof}

\section{The case of $\omega$-free PAC fields}

In this section apply the result of the previous section to the case of a group definable in an $\omega$-free PAC field.


\subsection{Some facts about e-free PAC fields}

We begin by recalling some results of Chatzidakis and Ramsey about measures and definable groups in $e$-free PAC fields when $e$ is finite, these results are from \cite{chatzidakis2023measures}. On this topic we can also quote a more recent article of Chatzidakis and Ramsey, \cite{chatzidakis2025measures}, which generalizes these results to bounded perfect PAC fields.

\begin{prop} \cite[Corollary of 4.5]{chatzidakis1998generic}\label{pacchar0} Let $F$ be a PAC field of characteristic $0$. If $E\subseteq F$, then $acl(E)$ is obtained by taking the relative field theoretic algebraic closure in $F$.
\end{prop}

Let $k$ be an $e$-free PAC field of characteristic $0$. In \cite{chatzidakis2023measures} Chatzidakis and Ramsey define a measure on the definable subsets of an absolutely irreducible variety. We begin by recalling the definition of this measure. Let $V$ be an absolutely irreducible variety defined over $k$.

\begin{definition} Let $E$ be a subfield of $k$. A test formula $\theta(x)$ over $E$ is a boolean combination of formulas of the form $\exists y (P(x,y)=0 )$ with $y$ a single variable and $P\in E[x,y]$.
\end{definition}

\begin{prop} \cite[Corollary 2.4]{chatzidakis2023measures} Let $k$ be a perfect $e$-free PAC field and $E\subseteq k$ be a subfield of $k$. Every formula over a field $E$ is equivalent in $k$ to a test formula over $E$.
\end{prop}

\begin{definition} \cite[Definition 3.1]{chatzidakis2023measures} Let $V$ be an absolutely irreducible variety defined over an $e$-free PAC field $k$ and let $a$ be a generic of $V$ over $k$. We define the measure $\mu_{V}$ on definable subsets of $V$ in $k$. Let $\varphi(x,b)$ be a formula over $k$. We can assume that $\varphi$ is a test formula. Let $L$ be the splitting field over $k(a)$ of all the polynomials appearing inside of $\varphi(x,b)$. Whenever some field $K$ contains $k(a)$, the satisfaction of the formula $\varphi(a,b)$ inside of $K$ only depends on the $k(a)$ isomorphism type of $K \cap L$.

\vspace{10pt}
The measure consists of counting the intermediate fields $K$ up to isomorphism above $k(a)$ such that $k(a)\subseteq K \subseteq L$ and that are realized as $\mathcal{K}\cap L$ for $k\prec \mathcal{K}$. We will call such fields \emph{permissible}. A field $K$ is permissible if and only if $K/k$ is a regular extension and $\Gal(L/K)$ is an image of $\mathcal{G}(k)$ (in our case this is equivalent to saying that $\Gal(L/K)$ is $e$-generated). Then $\mu_{V}(\varphi(x,b))$ is the proportion of the permissible fields (up to $k(a)$ isomorphism) in which $\varphi(a,b)$ is satisfied.\end{definition}

Since the field $k$ is PAC notice that such a generic $a$ can be found inside of an elementary extension of $k$.

\begin{definition}
Let $G$ and $H$ be two definable groups in some structure.  
We say that $G$ and $H$ are \emph{definably isogenous} if there exist a definable subgroup $S$ of $G \times H$ such that
\begin{itemize}
    \item The projection of $S$ to $G$, denoted by $G_{S}$, has a finite index in $G$,
    \item The projection of $S$ to $H$, denoted by $H_{S}$, has a finite index in $H$,
    \item The kernel $\ker(S) = \lbrace g \in G$ : $(g,1) \in S\rbrace$ and the co-kernel $\operatorname{coker}(S) = \lbrace h \in H $ : $ (1,h) \in S \rbrace$ are finite.
\end{itemize}

This subgroup $S$ is then an isogeny between $G$ and $H$. If $S$ is an isogeny between $G$ and $H$ then there is an isomorphism between 
$G_S / \ker(S)$ and $H_S / \operatorname{coker}(S)$.
\end{definition}

\begin{theorem}\cite[Theorem C]{hrushovski1994groups}\label{efreefini} Let $G$ be a group definable in a perfect $e$-free PAC-field. Then there is an algebraic group $H$ defined over $F$ and a definable (in $F$) isogeny between $G$ and $H(F)$.
\end{theorem}

\begin{remark} The above theorem is stated for pseudo-finite fields in \cite{hrushovski1994groups}, Hrushovski and Pillay then state that it also holds for perfect PAC fields whose absolute Galois group is finitely generated as a profinite group.
\end{remark}

Chatzidakis and Ramsey apply \cref{efreefini} and the isogeny to move the measure $\mu_{H}$ they defined on $H$ into $G$ and prove the following:  

\begin{theorem} Let $k$ be a perfect $e$-free PAC field, and let $G$ be a group definable in $k$. Then $G$ is definably amenable.
\end{theorem}

The remainder of this subsection is dedicated to showing that the measure $\mu_{V}$ is definable over models. For this we begin by a quick proof to get a good understanding of this measure.

\begin{prop} Let $k$ be an $e$-free PAC field and $V$ be an absolutely irreducible variety defined in $k$. The measure $\mu_{V}$ is invariant under elementary extension: Let $k\prec k'$ be an elementary extension, $b\in k$ and $\varphi(x,y)$ be a test formula. Then if we write $\mu_{V}'$ for the measure on the definable subsets of $V$ computed in $k'$ we have $\mu_{V}(\varphi(x,b))=\mu_{V}'(\varphi(x,b))$.
\end{prop}

\begin{proof} Let $k\prec k'$ be an elementary extension, $b\in k$ and $\varphi(x,y)$ be a test formula. We write $\mu_{V}'$ for the measure on the definable subsets of $V$ in the sense of $k'$. Let $a$ be a generic of $V$ over $k'$. Since $V$ is defined over $k$ we have that $a\forkindep^{ACF}_{k}k'$. Let $L$ be the splitting field of the polynomials appearing in $\varphi(a,b)$ over $k(a)$.

\vspace{10pt}
Define $L':=Lk'$. $L'/k'(a)$ is a finite Galois extension that contains all of the roots of the polynomials appearing in $\varphi(a,b)$. $L\forkindep^{ACF}_{k}k'$ and $k'/k$ is a regular extension, so $L$ and $k'$ are linearly disjoint over $k$. This implies that the Galois groups $\Gal(L/k(a))$ and $\Gal(L'/k'(a))$ are isomorphic via the restriction map, and that any intermediate extension $k'(a)\subseteq K' \subseteq L'$ can be written $K'=k'(K'\cap L)$. We show that an intermediate extension $k(a)\leq K\leq L$ is permissible if and only if $k'(a)\leq k'K\leq L'$ is permissible.

\vspace{10pt}
Let $k(a)\leq K \leq L$ be a permissible field. Define $K':=k'K$. $K/k$ is a regular extension, so $K'/k'$ and $K'/K$ both are regular extensions. Since $\Gal(L/K)$ is $e$-generated and since there is a surjective map from this group to $\Gal(L'/K')$ the group $\Gal(L'/K')$ is $e$-generated, and $k'(a)\leq K'\leq L'$ is permissible.

\begin{figure}[hbtp]
    \centering
\begin{tikzcd}
                                & k'K                                                           &                                  \\
K \arrow[ru, "reg" description] &                                                               & k' \arrow[lu, "reg" description] \\
                                & k \arrow[ru, "reg" description] \arrow[lu, "reg" description] &                                 
\end{tikzcd}
\caption{$K\forkindep^{ACF}_{k}k'$, $K/k$ and $k'/k$ are regular}

\end{figure}

\vspace{10pt}
Conversely assume that $K'$ is a permissible intermediate extension $k'(a) \subseteq K' \subseteq L'$. Let $K:=K'\cap L$, then $K'=k'K$. $K'/k$ is regular, so $K/k$ also is. 
Since $L\forkindep^{ACF}_{K}K'$ and $K'/K$ is regular $L$ and $K'$ are linearly disjoint over $K$. This implies that the restriction map $\Gal(L'/K')\rightarrow \Gal(L/K)$ is surjective, so $k(a)\leq K \leq L$ is permissible. 

\vspace{10pt}
The extension $K'/K$ is regular, and $K$ contains the parameters of the formula $\varphi(a,b)$. This formula is satisfied in $K'$ if and only if it is satisfied in $K$ (since $\varphi$ is a test formula). From this we see that there is the same number of admissible fields in $L$ and $L'$, and that an admissible field $k(a)\leq K\leq L$ satisfies the formula $\varphi$ if and only if $K':=k'K$ satisfies it, so $\mu_{V}(\varphi(x,b))=\mu_{V}'(\varphi(x,b))$.
\end{proof}

We now recall and prove a few facts about field extensions which we will need for the proof of \cref{defsurk}.

\begin{remark}\label{equalextens} Let $k$ be an $e$-free PAC field of characteristic $0$, $V$ an absolutely irreducible variety defined over some subfield $v \subseteq k$ and $a=a_{1}..a_{k}$ be a generic of $V$ over $k$. 
Consider $P(X) \in k[X]$ a polynomial in a tuple over $k$ in a tuple of variable $X$, with $a$ an $X$-tuple. Then $P(a)=0$ if and only if $P(x)=0$ for all $x\in V(k)$. This means that the equality inside of $k(a)$ between two polynomials $P(a),Q(a)\in k(a)$ is definable by a formula over $v$ in their coefficients.\end{remark}

\begin{cor}\label{polirrka} Let $k$ be an $e$-free PAC field of characteristic $0$, $V$ an absolutely irreducible variety defined over some subfield $v \subseteq k$ and $a$ be a generic of $V$ over $k$. Let $x,y$ be two tuples of variables such that $a$ is an $x$-tuple. Given a sequence of rational fractions $(f_{i}(x,y))_{i\leq n}\in \mathbb{Q}(x,y)$ there is a formula $\theta(y)$ over $v$ such that $k\models \theta(b)$ if and only if $\sum_{i\leq n} f_{i}(a,b)\cdot X^{i} \in k(a)[X]$ is irreducible in $k(a)[X]$ for all $y$-tuples $b\in k$.
\end{cor}

\begin{proof} To state that a polynomial is irreducible we need to quantify over the polynomials of smaller degree. The problem here is that this case the coefficients of such polynomials are in $k(a)$. However we can find a bound on the degree of the coefficients in $a$:

\vspace{10pt}
Let $P = \sum_{i\leq m} g_{i}(a)\cdot X^{i} \in k(a)[X]$ where $g_{i}=\frac{\sum_{j\leq n_{i}}\lambda_{i,j}\cdot \prod_{l}a_{l}^{n_{l,i,j}}}{\sum_{j\leq m_{i}}\mu_{i,j} \cdot \prod_{l}a_{l}^{m_{i,j}}}$ with $\lambda_{i,j},\mu_{i,j}\in k$ and the polynomials $\sum_{j\leq n_{i}}\lambda_{i,j}\cdot \prod_{l}X_{l}^{n_{l,i,j}}$ and $\sum_{j\leq m_{i}}\mu_{i,j} \cdot \prod_{l}X_{l}^{m_{i,j}}$ are coprime in $k[\overline{Y}]$ where $\overline{Y}=Y_{1}..Y_{k}$.

\vspace{10pt}
We define the $a$-degree of $P$ as:

\begin{center}$\deg_{a}(P)=max\lbrace \sum_{l}( n_{l,i,j}) + \sum_{l} (m_{l,i,j})$ : $ i\leq m, j\leq m_{i}\rbrace$.
\end{center}

By a straightforward induction of the degree of $P$ we can show that if $P$ is a unitary polynomial in $k(a)[X]$ and if $Q$ is also unitary and divides $P$ then $\deg_{a}(Q) < l$ for some integer $l$ depending only on $\deg_{a}(P)$ and $\deg(P)$. With this in mind we can state that $\sum_{i\leq n} f_{i}(a,b)\cdot X^{i} \in k(a)[X]$ is irreducible in $k(a)[X]$ by stating:

\begin{center}
    $Q \times S \centernot = \sum_{i\leq n} f_{i}(a,b)\cdot X^{i}$ for all unitary $Q,S\in k(a)[X]$ such that $\deg(Q),\deg(S)< \deg(P)$ and $\deg_{a}(Q),\deg_{a}(S)<l$.
\end{center}

This is indeed a formula over $v$ since $Q\times S \centernot = \sum_{i\leq n} f_{i}(a,b)\cdot X^{i}$ is a disjunction of inequalities in $k(a)$ which are expressible by formulas over $v$ by \cref{equalextens}.
\end{proof}

\begin{remark}\label{extalg} Let $F$ be a field, $F(\overline{t})$ be a purely transcendental extension of $F$ with $\overline{t}=t_{1}..t_{m}$ a transcendence basis and let $L=F(\overline{t})(\alpha)$ be an algebraic extension of degree $n$ of $F(\overline{t})$. Suppose that $P(X)=  X^{n} + \sum_{i<n}\lambda_{i}(\overline{t}) \cdot X_{i}$ is the minimal polynomial of $\alpha$ over $F(\overline{t})$. Then there exists some integer $M$ depending only on $n=\deg(P)$ and on $\deg_{\overline{t}}(P)$ such that:

\begin{center}
$F^{alg}\cap L \subseteq \Lin_{F} \big\lbrace ( \prod_{i} t_{i}^{k_{i}}) \cdot \alpha^{j}$ : $k_{i} < M$ for all $i\leq m$ and $j<n \big\rbrace$.\end{center}
    
\end{remark}

\begin{proof}
The extension $F(\overline{t})/F$ is purely transcendental, so it is regular. The extension $L\cap F^{alg}/F$ is algebraic, so $F(\overline{t})$ and $L\cap F^{alg}$ are linearly disjoint over $F$. From this we deduce that the extension $L\cap F^{alg}/F$ has degree $\leq n$: Indeed any family of elements of $L\cap F^{alg}$ linearly independent over $F$ is linearly independent over $F(\overline{t})$, so it should have size $\leq n = [F(\overline{t})(\alpha):F(\overline{t})]$.

\vspace{10pt}
So if an element of $L$ is algebraic over $F$ it is witnessed by a polynomial of degree $\leq n$ in $F[X]$. Also: \begin{center} $\mathcal{B}:=\big\lbrace ( \prod_{i} t_{i}^{k_{i}}) \cdot \alpha^{j}$ : $\overline{k}\in \omega^{m}$ and $j<n \big\rbrace$ \end{center}

is a basis of $L$ as a $F$-vector space. From these two facts it is easy to see that if a $t_{i}$ with too high exponent appears inside of some $x \in L$ when writing $x$ as a linear combination of the elements of $\mathcal{B}$ with coefficients in $F$, then $P(x)\centernot= 0$ for every polynomial $Q\in F[X]$ of degree $\leq n$.
\end{proof}


The proof of the following proposition was hinted to me by Ramsey:

\begin{prop}\label{defsurk} Let $F$ be a monster $e$-free PAC field and $V$ be an absolutely irreducible variety defined in $F$. Let $k\prec F$ be an elementary substructure over which $V$ is defined. Then the measure $\mu_{V}$ is definable over $k$.\end{prop}

\begin{proof}
Let $(P_{i}(X,x,y))_{1 \leq i\leq m}$ be the polynomials appearing in the test formula $\varphi(x,y)$. Let $a$ be a generic of $V$ over $F$. Let $b$ be a $y$-tuple. We write $k_{b}:=acl(k(b))=F\cap (k(b))^{alg}$. Let $L_{b}$ be the splitting field of the polynomials $P_{i}(X,a,b)\in k(a,b)[X]$ over $k_{b}(a)$. Since $k_{b}(a)\forkindep^{ACF}_{k_{b}}F$ and $F/k_{b}$ is regular we have $G:= \Gal(FL_{b}/F(a)) \cong \Gal(L/k_{b}(a))$.

\vspace{10pt}
\textbf{Claim 1:} There is a formula $\theta_{1}(y)\in \tp(b/k)$ such that $\Gal(FL_{b'}/F(a)) \cong G$ whenever $b'\models \theta_{1}(y)$.

\vspace{10pt}
\textit{Proof of Claim 1:} We can write $L_{b} = k_{b}(a)(\alpha_{b})$. Let $P_{0}(X)\in k_{b}(a)[X]$ be the minimal polynomial of $\alpha_{b}$ over $F(a)$, let $n_{0}$ be the degree of $P_{0}$. There exist some polynomials $p_{i,j}(Z)\in k_{b}(a)[Z]$ of degree $\leq n_{0}$ with $p_{0,0}(Z)=Z$ and some polynomial $q((Z_{i,j})_{i,j})\in k_{b}(a)[(Z_{i,j})_{i,j}]$ such that:
\begin{enumerate}
    \item $P_{0}(X)$ is irreducible in $F(a)[X]$
    \item $P_{i}(X,a,b)\equiv \prod_{j<n_{i}}(X-p_{i,j}(Z))$ mod $P_{0}(Z)$ for every $i\leq m$,
    \item $q((p_{i,j}(Z))_{i,j}) \equiv Z$ mod $P_{0}(Z)$
\end{enumerate}

This can be expressed by a formula $\theta(y)$ with parameters in $k$: The coefficients are in $k_{b}(a)$ so we can take the algebraic formulas that they satisfy over $k,b$, by \cref{polirrka} we can state that the polynomial we get is irreducible since it has the same $\deg$ and $\deg_{a}$ as the $P_{0}$ of $b$. The congruence relations are witnessed by some equalities inside of $k(a)$, and so inside of $k$ by \cref{equalextens}. Then the condition $\Gal(P_{0}(X),F(a)) \cong G$ as permutation groups can be expressed using these polynomials.

\vspace{10pt}
Let $z_{j}:= p_{0,j}(\alpha_{b})$ for $1 \leq j \leq l$. If $\sigma\in \Gal(P_{0}(X)/F(a))$ then there is a unique $\sigma' \in S_{n_{0}}$ such that $\sigma (z_{i})=z_{\sigma' (i)}$ so we can see $\Gal(P_{0}(X)/F(a))$ as a subgroup of $S_{n_{0}}$. Then $\Gal(P_{0}(X),F(a)) \cong G$ is equivalent to:

\begin{center}
$\bigwedge\limits_{\sigma \in G} \left(  \bigwedge\limits_{j=1}^{l} (p_{0,\sigma (j)}(z_{1}) = p_{0,j}(z_{\sigma (j)})) \right) \wedge \bigwedge\limits_{\sigma \in S_{n}\setminus G} \left( \bigvee\limits_{j=1}^{l} (p_{0,\sigma (j)}(z_{1}) \centernot= p_{0,j}(z_{\sigma (j)})) \right)$
\end{center}

These equalities and inequalities in $F(a)[X]/ (P_{0}(X))$ correspond to systems of equalities and inequalities in $F(a)$, so by \cref{equalextens} they correspond to formulas over $k$.\hspace*{0pt}\hfill\qedsymbol{}

\vspace{10pt}
Let $F_{L_{b}}:= FL_{b}\cap F^{alg}$ and $H:= \Gal(F_{L_{b}}/F)$.

\vspace{10pt}
\textbf{Claim 2:} There is a formula $\theta_{2}(y)\in \tp(b/k)$ strengthening $\theta_{1}$ such that $\Gal(F_{L_{b'}}/F) \cong H$ whenever $b'\models \theta_{2}(y)$.
\vspace{10pt}

\textit{Proof of Claim 2:} Let $\beta_{b} = q(\alpha_{b})$ be a generator of $F_{L_{b}}$ over $F$, with $q\in F(a)[X]$. Let $Q(X)\in F[X]$ be the minimal polynomial of $q(\alpha_{b})$ over $F$ and let $l:=\deg(Q)$. Let $P_{0}$ be the polynomial from Claim 1. There exist some polynomials $q_{i}(Z)\in F[Z]$ for $i<l$ of degree $\leq l$ with $q_{0}(Z)=Z$ such that:
\begin{enumerate}
    \item $Q(X)\in F[X]$ is irreducible in $F[X]$,
    \item $Q(X) \equiv \prod_{j<n_{i}}(X-q_{i}(q(Z)))$ mod $P_{0}(Z)$,
    \item For every $x = g(\alpha_{b}) \in FL_{b}\cap F^{alg}$ with $g\in F(a)[X]$ there is $f\in F[X]$ of degree $\leq l$ such that $g(Z) \equiv f(q(Z))$ mod $P_{0}(Z)$.
\end{enumerate}

These conditions are expressible by some formulas, for the first two ones it is clear. For the third one, in order to express that $ FL_{b}\cap F^{alg}\subseteq F(\beta_{b})$ it is enough to let $g$ range over the polynomials of degree $\leq n$ that also have $\deg_{a}$ smaller than $M$, where $M$ is the bound from \cref{extalg} - so there is only a finite number of coefficients in $F$ over which we have to quantify. As previously we can express that $\Gal(Q,F) \cong H$ as a permutation group on the roots $q_{i}(Z)\in F[Z]$.\hspace*{0pt}\hfill\qedsymbol{}

\vspace{10pt}
We have a formula expressing both $\Gal(FL_{b'}/F(a))$ and $\Gal(F_{L_{b'}}/F)$ as permutation groups. An intermediate field $F(a) \subseteq F'\subseteq FL_{b'}$ is regular over $F$ if and only if the restriction $\Gal(FL_{b'}/F') \rightarrow \Gal(F_{L_{b'}}/F)$ is surjective. The conditions defining the measure $\mu_{V}(\varphi(x,b))$ are expressible by formulas over $k$:

\vspace{10pt}
- A subgroup $G'\subseteq G$ corresponds to a permissible field $F'= Fix(G')$ if and only if the restriction $G' \rightarrow H$ is surjective and if $G'$ is $e$-generated.

\vspace{10pt}
- Given a subgroup $G'\subseteq G$ such that $\mathrm{Fix}(G')$ is permissible we can express that $\mathrm{Fix}(G')\models \varphi (a,b')$ by saying which roots of the $P_{i}$ are fixed by the elements of $G'$ and which are not.\end{proof}

\subsection{$\omega$-free PAC fields}

Let $F$ be an $\omega$-free PAC field. By \cite[Lemma 4.2]{chatzidakis2023measures} any $\omega$-free PAC is elementary equivalent to an ultraproduct of $e$-free PAC fields. Replacing $F$ by an elementarily equivalent field, we may assume that $F=\mathbb{F}:=\prod_{e<\omega}\mathbb{F}_{e}/\mathcal{U}$, where the $\mathbb{F}_{e}$ are $\kappa$-saturated $e$-free PAC field of characteristic $0$ for $e<\omega$, and $\mathcal{U}$ is a non-principal ultrafilter.

\vspace{10pt}
Let $G$ be a group definable in $\mathbb{F}$ over a finite set $A$. The group $G$ can be written as $G=\prod_{e<\omega}G_{e}/\mathcal{U}$ where $G_{e}$ is a group definable in $\mathbb{F}_{e}$ over some finite set of parameters $A_{e}$ for $\mathcal{U}$-almost all $e<\omega$.

\vspace{10pt}
Using the measures from the previous subsection Chatzidakis and Ramsey prove the following:

\begin{prop}\cite[Proposition 4.13]{chatzidakis2023measures} Let $G$ be a group definable in an $\omega$-free PAC field of characteristic $0$. Then $G$ is definably amenable.\end{prop}

In this context every group $G_{e}$ in $M_{e}$ is definably amenable with a measure $\mu_{e}$ which is invariant under translations, is such that $\mu(G_{e})=1$ and is definable over an $e$-free PAC field $A_{e} \subseteq M_{e}\prec \mathbb{F}_{e}$ by \cref{defsurk}. $G$ is also definably amenable with the ultraproduct measure $\mu = \prod_{e<\omega} \mu_{e}/\mathcal{U}$, which is also invariant under translation. Let $M:=\prod_{e<\omega}M_{e}/\mathcal{U} \prec \mathbb{F}$.

\vspace{10pt}
In the previous subsection we saw that the definable amenability of definable groups in $e$-free PAC fields with $e$ finite is proven using the fact that such a group is definably isogenous to the rational points of an algebraic group. This raises an insightful question as to whether the same applies to $\omega$-free PAC field. With the previous results from this chapter we get a partial answer:

\begin{theorem}\label{omegafreegroup} There exist an algebraic group definable over $M$, a type-definable subgroup $G_{0}$ of bounded index in $G$ and a definable finite to one group homomorphism from $G_{0}$ to $H(\mathbb{F})$.
\end{theorem}

\begin{proof} We apply \cref{resultatgroup}. By \cref{stabomegafree} the assumption (H1) is satisfied. (H2) follows from \cref{pacchar0}. We want to show that (H3) holds, for this we construct an expansion of the language.

\vspace{10pt}
By \cref{mesureprod} and \cref{mesureprod} there is a countable extension $L'$ of the language $L_{ring}$ and a canonical $L'$-structure on $\mathbb{F}$ such that the group $G$ is definably amenable in $L'$ by a measure $\mu'$ extending $\mu$ which is definable over $M$ and invariant under translations. We will write $\mathbb{F}'$ for $\mathbb{F}$ seen as an $L'$-structure.

\vspace{10pt}
We have a definably amenable group in $L'$, so the ideal of formulas of measure $0$ is $S1$ on $G$, $M$-invariant and invariant under translation, so (H3) holds. By \cref{resultatgroup} there is a type-definable subgroup $G_{0}$ of $G$ of bounded index, an algebraic group $H$ and a definable finite to one group homomorphism $\varphi$ from $G_{0}$ to $H(\mathbb{F})$.\end{proof}

\begin{remark}
We can ask ourselves if we can get a stronger conclusion: In fact the subgroup $G_{0}$ is wide and the group homomorphism $\varphi$ from $G_{0}$ to $H(\mathbb{F})$ is defined on a domain $X_{\varphi}$, which is a definable set included in $G$ containing $G_{0}$. So $\mu(X_{\varphi}) > \frac{1}{n}$ for some $n<\omega$. This implies that any group containing $X_{\varphi}$ has finite index, and so finding an extension of $\varphi$ to such a group would strengthen the conclusion to $G_{0}$ being of finite index.\end{remark}

\begin{remark} \cref{omegafreegroup} does not hold for interpretable groups: In \cite{hrushovski1994groups} (p. 207) Hrushovski and Pillay point out that if $F$ is a PAC field of characteristic $0$ such that the subgroup of $n$-th powers have infinite index in $(F^{*},\cdot)$, then the group $G=F^{*}/(F^{*})^{n}$ can not be abstractly isomorphic to $H(F)$ for $H$ any algebraic group. More generally no type-definable subgroup of $G$ of bounded index can be sent by a morphism of finite kernel into $H(F)$.
\end{remark}

\section{Some questions}

The main difficulty we face when trying to prove some result on the Kim-stabilizer is that without additional assumptions it is difficult to show that the set $S^{K}(p)$ is not trivial, or that it is large in some sense for a certain type. In the proof of \cref{resultatgroup} we use the existence of the ideal $\mu$ to know that the type defining $S^{K}(p)$ is $\mu$-wide.

\vspace{10pt}
We conclude this chapter with some questions that we collected from the different parts of this work:

\begin{enumerate}
\item Is there an interesting notion of genericity in groups definable in NSOP$_1$ theories?
    \item Is $Stab^{K}(p)$ type-definable for every type $p$? Is it always equal to $S^{K}(p)\cdot S^{K}(p)$?
    \item More accurately, if $X$ is a definable or type-definable subset of a group $G$ definable in an NSOP$_1$ theory that is stable by Kim-independent product (over the parameters defining the group), is the subgroup generated by $X$ definable or type-definable?
    \item When does $S^{K}(p)$ contain a translate of $p$? 
    \item Consider a measure $\mu$ on definable subsets of a definable group $G$ which is invariant by translation in a $L$-theory $T$. Are there some conditions under which there exist an extension $T'$ of $T$ in some language $L'$ and an extension of $\mu$ into a measure $\mu'$ on $L'$-definable subsets of $G$ that is definable and also invariant by translation?
\end{enumerate}

\bibliographystyle{alpha}
\bibliography{Ref}

\end{document}